\documentclass[12pt]{amsart}
\usepackage{amssymb}

\let\cal\mathcal
\let\frak\mathfrak
\let\Bbb\mathbb

\usepackage{eucal}

\usepackage{graphicx}

\def\>{\relax\ifmmode\mskip.666667\thinmuskip\relax\else\kern.111111em\fi}
\def\:{\relax\ifmmode\mskip.333333\thinmuskip\relax\else\kern.0555556em\fi}
\def\<{\relax\ifmmode\mskip-.333333\thinmuskip\relax\else\kern-.0555556em\fi}
\def\?{\relax\ifmmode\mskip-.666667\thinmuskip\relax\else\kern-.111111em\fi}
\def\vsk#1>{\vskip#1\baselineskip}
\def\vv#1>{\vadjust{\vsk#1>}\ignorespaces}
\def\vvn#1>{\vadjust{\nobreak\vsk#1>\nobreak}\ignorespaces}
 \let\alb\allowbreak

\def\plait#1{\par\hangindent2\parindent\indent\kern\parindent
\llap{#1\enspace}\ignorespaces}

\let\Smallskip\smallskip
\def\smallskip{\par\Smallskip}
\let\Medskip\medskip
\def\medskip{\par\Medskip}
\let\Bigskip\bigskip
\def\bigskip{\par\Bigskip}

\let\Maketitle\maketitle
\def\maketitle{\Maketitle\thispagestyle{empty}\let\maketitle\empty}

\newtheorem{thm}{Theorem}[section]
\newtheorem{cor}[thm]{Corollary}
\newtheorem{lem}[thm]{Lemma}
\newtheorem{prop}[thm]{Proposition}

\newtheorem{defn}[thm]{Definition}

\numberwithin{equation}{section}

\theoremstyle{definition}
\newtheorem*{rem}{Remark}

\def\beq{\begin{equation}}
\def\eeq{\end{equation}}
\def\be{\begin{equation*}}
\def\ee{\end{equation*}}

\def\bean{\begin{eqnarray}}
\def\eean{\end{eqnarray}}
\def\bea{\begin{eqnarray*}}
\def\eea{\end{eqnarray*}}

 \let\eps\varepsilon \let\epsilon\eps

\let\la\lambda

 \let\phi\varphi

\let\longto\longrightarrow

\let\ge\geqslant
\let\geq\geqslant
\let\le\leqslant
\let\leq\leqslant

\def\C{\Bbb C}
\def\Z{\Bbb Z}

\def\Dc{\cal D}

\def\hg{\frak h}

\def\gl{\frak{gl}}

\def\lsym#1{#1\alb\dots\relax#1\alb} \def\lc{\lsym,}

\let\on\operatorname

\def\End{\on{End}}

\def\rdet{\on{rdet}}
\def\Res{\on{Res}}

\def\Bh{\widehat B}
\def\Ch{\widehat C}
\def\Dh{\widehat D}
\def\Vh{\widehat V}

\def\Dt{\acute D}
\def\Dti{\widetilde D}
\def\ev{\text{\sl ev}}

\def\KZ/{{\sl KZ\/}}
\def\qKZ/{{\sl qKZ\/}}

\def\red{\mathrm{red}}
\def\aug{\mathrm{aug}}

\begin{document}

\hrule width0pt
\vsk->

\title[Duality for Bethe algebras]
{Duality for Bethe algebras acting on polynomials in anticommuting variables}

\author[V\<.\,Tarasov]{V\<.\,Tarasov$\:^\circ$}
\thanks{$\kern-\parindent^\circ$E\:-mail:
vtarasov@iupui.edu, vt@pdmi.ras.ru\\
$^\circ$Supported in part by Simons Foundation grant 430235 and
RFBR grant 18\:-\:01-\:00271\>}

\author[F\<.\,Uvarov]{F\<.\,Uvarov$\:^\star$}
\thanks{\noindent$^\star$E\:-mail: filuvaro@iu.edu}

\maketitle

\begin{center}
\vsk-.2>
{\it $^{\circ\:\star}\?$Department of Mathematical Sciences,
Indiana University\,--\>Purdue University Indianapolis\kern-.4em\\
402 North Blackford St, Indianapolis, IN 46202-3216, USA\/}

\medskip
{\it $^\circ\?$St.\,Petersburg Branch of Steklov Mathematical Institute\\
Fontanka 27, St.\,Petersburg, 191023, Russia\/}
\end{center}

\begin{abstract}
We consider actions of the current Lie algebras $\gl_{n}[t]$ and $\gl_{k}[t]$
on the space of polynomials in $kn$ anticommuting variables. The actions depend
on parameters $\bar{z}=(z_{1}\lc z_{k})$ and
$\bar{\alpha}=(\alpha_{1}\lc\alpha_{n})$, respectively.
We show that the images of the Bethe algebras
$\mathcal{B}_{\bar{\alpha}}^{\langle n \rangle}\subset U(\gl_{n}[t])$ and
$\mathcal{B}_{\bar{z}}^{\langle k \rangle}\subset U(\gl_{k}[t])$ under these
actions coincide. To prove the statement, we use the Bethe ansatz description
of eigenvalues of the actions of the Bethe algebras via spaces of
quasi-exponentials and establish an explicit correspondence between these
spaces for the actions of $\mathcal{B}_{\bar{\alpha}}^{\langle n \rangle}$
and $\mathcal{B}_{\bar{z}}^{\langle k \rangle}$.
\end{abstract}

\section{Introduction}
The classical $(\gl_{k},\gl_{n})$-duality plays an important role in the representation theory and the classical invariant theory,
for example, see \cite{H}, \cite{Zh}. It states the following. Let $e^{\langle n\rangle}_{ij}$, $i,j=1\lc n$, and $e^{\langle k\rangle}_{ab}$, $a,b=1\lc k$, be the standard generators of the Lie algebras $\gl_{n}$ and $\gl_{k}$, respectively. Define $\gl_{n}$- and $\gl_{k}$-actions on the space $P_{kn}=\C [x_{11}\lc x_{kn}]$ of polynomials in $kn$ variables:
\begin{equation}\label{act1}
e^{\langle n\rangle}_{ij}\mapsto\sum_{a=1}^{k}x_{ai}\frac{\partial}{\partial x_{aj}},
\end{equation}
\begin{equation}\label{act2}
e^{\langle k\rangle}_{ab}\mapsto\sum_{i=1}^{n}x_{ai}\frac{\partial}{\partial x_{bi}}.
\end{equation}
Then actions \eqref{act1} and \eqref{act2} commute and there is an isomorphism of $\gl_{k}\oplus\gl_{n}$-modules
\begin{equation}\label{module decomposition}
P_{kn}\cong\bigoplus_{\lambda}V_{\lambda}^{\langle n\rangle}\otimes V_{\lambda}^{\langle k\rangle},
\end{equation}
where $V_{\lambda}^{\langle n\rangle}$ and $V_{\lambda}^{\langle k\rangle}$ are the irreducible representations of $\gl_{n}$ and $\gl_{k}$ of highest weight $\lambda$, respectively.

It is interesting to study a similar duality in the context of current algebras, where the central role is played by the commutative subalgebras $\mathcal{B}_{\bar{\alpha}}^{\langle n \rangle}\subset U(\gl_{n}[t])$ and $\mathcal{B}_{\bar{z}}^{\langle k \rangle}\subset U(\gl_{k}[t])$, called the Bethe algebras. They depend on parameters $\bar{\alpha}=(\alpha_{1}\lc\alpha_{n})$ and $\bar{z}=(z_{1}\lc z_{k})$, respectively.
One can extend the actions of $\gl_{n}$ and $\gl_{k}$ on $P_{kn}$
to the respective $\gl_{n}[t]$- and $\gl_{k}[t]$-actions by the following formulas:
\begin{equation}\label{act3}
\psi_{\bar{z}}^{\langle n\rangle}: e^{\langle n\rangle}_{ij}\otimes t^{s}\mapsto\sum_{a=1}^{k}z_{a}^{s}x_{ai}\frac{\partial}{\partial x_{aj}},
\end{equation}
\begin{equation}\label{act4}
\psi_{\bar{\alpha}}^{\langle k\rangle}: e^{\langle k\rangle}_{ab}\otimes t^{s}\mapsto\sum_{i=1}^{n}\alpha_{i}^{s}x_{ai}\frac{\partial}{\partial x_{bi}}.
\end{equation}
Actions \eqref{act3} and \eqref{act4} do not commute anymore. However, the images of the subalgebras $\mathcal{B}_{\bar{\alpha}}^{\langle n\rangle}$ and $\mathcal{B}_{\bar{z}}^{\langle k\rangle}$ under the corresponding actions coincide, see \cite{MTV1}.

According to \cite{MTV2}, the Bethe ansatz method gives a bijection between eigenvectors of the action of $\mathcal{B}_{\bar{\alpha}}^{\langle n \rangle}$ on $P_{kn}$ and $n$-th order monic ordinary differential operators whose kernels are certain spaces of quasi-exponentials, see Section 3. Similarly, there is a bijection between eigenvectors of the action of $\mathcal{B}_{\bar{z}}^{\langle k \rangle}$ on $P_{kn}$ and $k$-th order monic ordinary differential operators of the same kind. Since the images of $\mathcal{B}_{\bar{\alpha}}^{\langle n \rangle}$ and $\mathcal{B}_{\bar{z}}^{\langle k \rangle}$ acting on $P_{kn}$ coincide and, therefore, have the same eigenvectors, this yields a bijection between the sets of the corresponding ordinary differential operators. It was conjectured in \cite{MTV3} that this bijection for differential operators is the bispectral duality. Namely, there is an automorphism $(\cdot)^{\ddagger}$ of the algebra of differential operators such that an eigenvector $v\in P_{kn}$ of the action $\psi_{\bar{z}}^{\langle n\rangle}$ of $\mathcal{B}_{\bar{\alpha}}^{\langle n\rangle}$ corresponding to a differential operator $D$ is an eigenvector of the action $\psi_{\bar{\alpha}}^{\langle k\rangle}$ of $\mathcal{B}_{\bar{z}}^{\langle k\rangle}$ corresponding to the differential operator $D^{\ddagger}$.

One can also study the $(\gl_{k},\gl_{n})$-duality for different representations. For example, instead of $P_{kn}$, one can consider the space $\mathfrak{P}_{kn}$ of polynomials in $kn$ anticommuting variables $\xi_{11}\lc\xi_{kn}$ with the actions of $\gl_{n}$ and $\gl_{k}$ similar to those \eqref{act1}, \eqref{act2} on $P_{kn}$. The analog of isomorphism \eqref{module decomposition} is given by:
\begin{equation*}
\mathfrak{P}_{kn}\cong\bigoplus_{\lambda}V^{\langle n\rangle}_{\lambda}\otimes V^{\langle k\rangle}_{\lambda '},
\end{equation*}
where $\lambda'$ denotes the conjugate partition, see \cite{CW}. Similarly to the case of $P_{kn}$, the $\gl_{n}$- and $\gl_{k}$-actions on $\mathfrak{P}_{kn}$ extend to $\gl_{n}[t]$- and $\gl_{k}[t]$-actions $\pi_{\bar{z}}^{\langle n\rangle}$ and $\pi_{-\bar{\alpha}}^{\langle k\rangle}$, depending on parameters $\bar{z}$ and $\bar{\alpha}$, respectively. The actions $\pi_{\bar{z}}^{\langle n\rangle}$ and $\pi_{-\bar{\alpha}}^{\langle k\rangle}$ are given essentially by the same formulas as for $\psi_{\bar{z}}^{\langle n\rangle}$ and $\psi_{\bar{\alpha}}^{\langle k\rangle}$, see Section 5. In this work, we will prove the equality of the images of the Bethe algebras $\mathcal{B}_{\bar{\alpha}}^{\langle n\rangle}$ and $\mathcal{B}_{\bar{z}}^{\langle k\rangle}$ under the actions $\pi_{\bar{z}}^{\langle n\rangle}$ and $\pi_{-\bar{\alpha}}^{\langle k\rangle}$, respectively, see Theorem \ref{main}.

Unlike \cite{MTV1}, our proof of Theorem \ref{main} is not straightforward. Instead, we explore the Bethe ansatz description of eigenvalues and eigenvectors of the actions $\pi_{\bar{z}}^{\langle n\rangle}$ of $\mathcal{B}_{\bar{\alpha}}^{\langle n\rangle}$ and $\pi_{-\bar{\alpha}}^{\langle k\rangle}$ of $\mathcal{B}_{\bar{z}}^{\langle k\rangle}$. Let $v\in\mathfrak{P}_{kn}$ be an eigenvector of the action $\pi_{\bar{z}}^{\langle n\rangle}$ of $\mathcal{B}_{\bar{\alpha}}^{\langle n\rangle}$ corresponding to a differential operator $D$. We introduce a transformation $D\to \Dti$, see formula \eqref{tilde D}, where $\Dti$ is another differential operator. This transformation can be naturally described using pseudodifferential operators, and the automorphism $(\cdot)^{\ddagger}$ appears as a step in the definition of this transformation. More explicit construction of $\Dti$ employs the so-called quotient differential operator, see Section 6. We prove that the vector $v\in\mathfrak{P}_{kn}$ is an eigenvector of the action $\pi_{-\bar{\alpha}}^{\langle k\rangle}$ of $\mathcal{B}_{\bar{z}}^{\langle k\rangle}$ corresponding to the differential operator $\Dti$, see Theorem \ref{main2}. This observation together with some results on eigenvalues and eigenvectors of the Bethe algebra from \cite{MTV2} allows us to prove Theorem \ref{main}.

There are important elements of the Bethe algebras called the Gaudin and Dynamical Hamiltonians. The exchange of the Gaudin and Dynamical Hamiltonians under the $(\gl_{k}, \gl_{n})$-duality for the space $P_{kn}$ was observed in \cite{TV}.
A similar result for the space $\mathfrak{P}_{kn}$ was obtained recently in \cite{TU}. The $(\gl_{k},\gl_{n})$-duality for the Gaudin and Dynamical Hamiltonians is an important step in our proof of Theorem \ref{main}.

The duality of the Bethe algebras of the Gaudin model considered in this paper is expected to extend to the duality of the Yangian Bethe algebras and the Bethe algebras of the trigonometric Gaudin model. This is currently work in progress. The construction of the quotient difference operator appearing in this generalization resembles the factorization of difference operators used in \cite{K} in connection with the combinatorial Gale transform introduced in \cite{MOST}.

In a recent paper \cite{HM}, the authors considered the duality of $\gl_{k}$ and $\gl_{m|n}$ Gaudin models and established a generalization of Theorem \ref{main} to that case. Their proof is similar to that of Capelli type identity in \cite{MTV1}. It would be interesting to extend our approach to the case of \cite{HM} to have an appropriate relation between differential and rational pseudodifferential operators.

The $(\gl_{k}, \gl_{n})$-duality for classical integrable models related to Gaudin Hamiltonians and the actions of $\gl_{k}$ and $\gl_{n}$ on the space of polynomials in anticommuting variables was studied in \cite[Section 3.3]{VY}. The result of \cite{VY} resembles the construction of the differential operator $\Dti$ discussed in our work.

The paper is organized as follows. In Section \ref{s2}, we introduce the algebra of pseudodifferential operators. In Section \ref{s3}, we describe the spaces of quasi-exponentials, define the transformation $D\to\Dti$, and formulate important properties of the kernel of $\Dti$, see Theorem \ref{1}.
We recall the definition of the Bethe algebra
of the Gaudin model
and some results
about its action on finite-dimensional irreducible representations
of the current Lie algebra
in Section \ref{s4}.
In Section \ref{s5}, we discuss the $(\gl_{k}, \gl_{n})$-duality for the space $\mathfrak{P}_{kn}$, formulate and prove the main result, Theorem \ref{main}. In Section \ref{s6}, we introduce the quotient differential operator and give a proof of Theorem \ref{1}.

\section{The Algebra of Pseudodifferential Operators}
\label{s2}
The algebra of pseudodifferential operators $\Psi\mathfrak{D}$
consists
of all formal series of the form $$\sum_{m=-\infty}^{M}\sum_{k=-\infty}^{K}C_{km}x^{k}\left(\frac{d}{dx}\right)^{m},$$ where integers $M$ and $K$
can differ
for different series, and $C_{km}$ are complex numbers. One can check that the rule
\begin{equation}\label{algebra relation}
\left(\frac{d}{dx}\right)^{m}x^{k}=\sum_{j=0}^{\infty}\frac{(m)_{j}(k)_{j}}{j!}
\;x^{k-j}\left(\frac{d}{dx}\right)^{m-j},
\qquad m,k\in\Z\,,\kern-2em
\end{equation}
where $\;(a)_{i}=a(a-1)(a-2)\dots (a-i+1)\,$,
yields a well-defined multiplication on $\Psi\mathfrak{D}$.
The verification of associativity is straightforward using the Chu-Vandermonde identity:
$$
\sum_{j=1}^{i}\frac{(m-n)_{j}}{j!}\cdot\frac{(n)_{i-j}}{(i-j)!}\,=\,
\frac{(m)_{i}}{i!}\;.
$$
\begin{lem}\label{invertible pseudodifferential operators}
If $\,D=\sum_{m=-\infty}^{M}\sum_{k=-\infty}^{K}C_{km}\,x^{k}(d/dx)^{m}$, with $\,C_{KM}\neq 0$, then $D$ is invertible in $\Psi\mathfrak{D}$.
\end{lem}
\begin{proof}
Define $\Dt$ by the rule $\,1+\Dt=C_{KM}^{-1}\,x^{-K}D\,(d/dx)^{-M}$.
\vvn.1>
Then $\sum_{j=0}^{\infty}(-1)^{j}\Dt^{j}$ is a well-defined element of $\Psi\mathfrak{D}$ and the inverse of $D$ is given by the formula:
$$
D^{-1}=C_{KM}^{-1}\left(\frac{d}{dx}\right)^{-M}\left(\sum_{j=0}^{\infty}(-1)^{j}\Dt^{j}\right)x^{-K}.
\vvn-1.2>
$$
\end{proof}

We consider a formal series $\sum_{m=-\infty}^{M}f_{m}(x)(d/dx)^{m}$, where all $f_{m}(x)$ are rational functions, as an element of $\Psi\mathfrak{D}$ replacing each $f_{m}(x)$ by its Laurent series at infinity. In particular, we identify the algebra of linear differential operators with rational coefficients and the corresponding subalgebra of $\Psi\mathfrak{D}$.
Next corollary follows immediately from the Lemma \ref{invertible pseudodifferential operators}.

\begin{cor}\label{invertible2}
Let $D=\sum_{m=-\infty}^{M}f_{m}(x)(d/dx)^{m}$, where all $f_{m}(x)$ are
\vvn.1>
rational functions regular at infinity. Then
$D$ is invertible in $\Psi\mathfrak{D}$.
\end{cor}
Using \eqref{algebra relation}, one can check that for any complex numbers $C_{km}$, the series $$\sum_{m=-\infty}^{M}\sum_{k=-\infty}^{K}C_{km}\left(-\frac{d}{dx}\right)^{m}x^{k}$$ is a well-defined element of $\Psi\mathfrak{D}$. We define a map $(\cdot)^{\dagger}:\Psi\mathfrak{D}\to \Psi\mathfrak{D}$ by the rule
$$\left(\sum_{m=-\infty}^{M}\sum_{k=-\infty}^{K}C_{km}\,x^{k}\left(\frac{d}{dx}\right)^{m}\right)^{\dagger}=\sum_{m=-\infty}^{M}\sum_{k=-\infty}^{K}C_{km}\left(-\frac{d}{dx}\right)^{m}x^{k}.$$
\begin{lem}\label{antiauto}
The map $(\cdot)^{\dagger}$ is an involutive antiautomorphism of $\Psi\mathfrak{D}$.
\end{lem}
\begin{proof}
To check that $(\cdot)^{\dagger}$ is involutive,
we need to verify that $\left(\!\left(x^{k}(d/dx)^{m}\right)^{\dagger}\right)^{\dagger}=x^{k}(d/dx)^{m}$.
By \eqref{algebra relation}, it reads as
\begin{equation}\label{4}
\sum_{j=0}^{\infty}\,(-1)^{j}\,\frac{(k)_{j}(m)_{j}}{j!} \left(\frac{d}{dx}\right) ^{m-j}x^{k-j}=x^{k}\left(\frac{d}{dx}\right)^{m}.
\end{equation}
The equality holds since
\vvn->
$$
\sum_{j=0}^{i}\frac{(-1)^{j}}{j!(i-j)!}=\delta_{i0}\,.
$$
Using \eqref{algebra relation} and \eqref{4},
one can check that $(\cdot)^{\dagger}$ is an antiautomorphism as well.
\end{proof}

We also define the following involutive antiautomorphism on $\Psi \mathfrak{D}$:
$$\left(\sum_{m=-\infty}^{M}\sum_{k=-\infty}^{K}C_{km}\,x^{k}\left(\frac{d}{dx}\right)^{m}\right)^{\ddagger}=\sum_{m=-\infty}^{M}\sum_{k=-\infty}^{K}C_{km}\,x^{m}\left(\frac{d}{dx}\right)^{k}.$$

For $D\in \Psi\mathfrak{D}$, we say that $D^{\dagger}$ is the \textit{formal conjugate} to $D$ and $D^{\ddagger}$ is the \textit{bispectral dual} to $D$.
Let $D^{\#}=\left(D^{\dagger}\right)^{\ddagger}$.
\begin{lem}
The
map
$(\cdot)^{\#}$ is an automorphism on $\Psi\mathfrak{D}$ of order 4
\end{lem}
\begin{proof}
The
map
$(\cdot)^{\#}$ is an automorphism
because
it is a composition of two
antiautomorphisms.
Since $\left(\!{\left(x^{k}(d/dx)^{m}\right)^{\#}}\right)^{\!\#}=
(-x)^{k}\left(-d/dx\right)^{m}$, the map $(\cdot)^{\#}$ has order $4$.
\end{proof}

\section{The spaces of quasi-exponentials}
\label{s3}
In this paper, a partition $\mu=(\mu_{1},\mu_{2},\dots{})$ is an infinite nonincreasing sequence of nonnegative integers stabilizing at zero. Let $\mu'=(\mu'_{1},\mu'_{2},\dots{})$ denote the conjugate partition,
that is, $\mu'_{i}=\#\{j\;\vert\;\mu_{j}\geq i\}$.
In particular,
$\mu_{1}'$ equals the number of nonzero entries in $\mu$.

Fix complex numbers $\alpha_{1}\lc\alpha_{n}$ and nonzero partitions
$\mu^{(1)}\lc\mu^{(n)}$. Assume that $\alpha_{i}\neq\alpha_{j}$ for $i\neq j$.
Let $V$ be a vector space of functions in one variable with a basis
$\{q_{ij}(x)e^{\alpha_{i}x}\;|\alb\;i=1\lc n,\alb j=1\lc (\mu^{(i)})'_{1}\}$,
where $q_{ij}(x)$ are polynomials and $\deg q_{ij} =
(\mu^{(i)})'_{1}+\mu^{(i)}_{j}-j$.

Denote $M'=\sum_{i=1}^n(\mu^{(i)})'_{1}=\dim V$. For $z\in\C$, define \textit{the sequence of exponents of $V$ at $z$} as
a unique sequence of integers $\boldsymbol{e}=\{e_{1}>\ldots>e_{M'}\}$,
with the property:
for each $i=1\lc M'$, there exists $f\in V$ such that $f(x)=(x-z)^{e_i}\bigl(1+o(1)\bigr)$ as $x\to z$.

We say that $z\in\C$ is a singular point of $V$ if the set of exponents of $V$ at $z$ differs from the set $\{0\lc M'-1\}$. A space of quasi-exponentials has finitely many singular points.

Let $z_{1}\lc z_{k}$ be all singular points of $V$ and let $\boldsymbol{e}^{(a)}=\{e^{(a)}_{1}>\ldots>e^{(a)}_{M'}\}$ be the set of exponents of $V$ at $z_{a}$. For each $a=1\lc k$, define a partition $\lambda^{(a)}=(\lambda^{(a)}_{1},\lambda^{(a)}_{2},\dots{})$ as follows: $e^{(a)}_{i}=M'+\lambda^{(a)}_{i}-i$ for $i=1\lc M'$, and $\lambda^{(a)}_{i}=0$ for $i>M'$.
Clearly, all partitions $\lambda^{(1)}\!\lc\la^{(k)}$ are nonzero.

Denote the sequences $(\mu^{(1)}\lc\mu^{(n)})$, $(\lambda^{(1)}\lc\lambda^{(k)})$, $(\alpha_{1}\lc\alpha_{n})$, $(z_{1}\lc z_{k})$ as $\bar{\mu}$, $\bar{\lambda}$, $\bar{\alpha}$, $\bar{z}$, respectively. We will say that $V$ is a \textit{space of quasi-exponentials with the data $(\bar{\mu},\bar{\lambda} ;\bar{\alpha},\bar{z})$.}

For arbitrary sequences of partitions $\bar\mu=(\mu^{(1)}\lc\mu^{(n)})$,
$\bar\lambda=(\lambda^{(1)}\lc\lambda^{(k)})$, and sequences of complex numbers
$\bar\alpha=(\alpha_{1}\lc\alpha_{n})$, $\bar z=(z_{1}\lc z_{k})$, define
the data $(\bar{\mu},\bar{\lambda};\alb\bar{\alpha},\alb\bar{z})^{\red}$
by removing all zero partitions from the sequences $\bar\mu\,,\bar\lambda$
and the corresponding numbers from the sequences $\bar\alpha\,,\bar z$.
We will call the data
$(\bar{\mu},\bar{\lambda};\bar{\alpha},\bar{z})$ \textit{reduced\/} if
$(\bar{\mu},\bar{\lambda};\bar{\alpha},\bar{z})=(\bar{\mu},\bar{\lambda};\bar{\alpha},\bar{z})^{\red}$.

We will say that $V$ is a \textit{space of quasi-exponentials with the data\/}
$(\bar{\mu},\bar{\lambda};\bar{\alpha},\bar{z})$
if $V$ is a space of quasi-exponentials with the data
$(\bar{\mu},\bar{\lambda};\bar{\alpha},\bar{z})^{\red}$.

\textit{The fundamental operator of\/} $\,V$ is a unique monic linear differential operator of order $M'$ annihilating $V$. Denote the fundamental operator of $V$ by $D_V$.

\smallskip
Define \,$D_V^{\aug}=D_V\,\prod_{i=1,\,\mu^{(i)}=0}^n(d/dx-\alpha_i)$. We will
say that the space $V^{\aug}=\ker D_V^{\aug}$ is the \textit{augmentation
of $\,V\!$ with the data} $(\bar{\mu},\bar{\lambda};\bar{\alpha},\bar{z})$\,,
and the space $V$ is the \textit{reduction of} $\,V^{\aug}$.
Clearly, \,$V=\prod_{i=1,\,\mu^{(i)}=0}^n(d/dx-\alpha_i)\,V^{\aug}$.

\begin{lem}\label{coefficients}
The coefficients of $\,D_V$ and $D_V^{\aug}$ are rational functions in $\,x$ regular at infinity.
\end{lem}

\noindent
The lemma will be proved in Section \ref{s6}.

\smallskip
Recall that we identify the algebra of linear differential operators with
rational coefficients and the corresponding subalgebra of $\Psi\mathfrak{D}$.

Let $V$ be a \textit{space of quasi-exponentials with the data}
$(\bar{\mu},\bar{\lambda};\bar{\alpha},\bar{z})$. By Lemma \ref{coefficients}
and Corollary \ref{invertible2}, the operator $D_V$ is an invertible element of
$\Psi\mathfrak{D}$. Consider the following pseudodifferential operator:
\begin{equation}\label{tilde D}
\Dti_V=(-1)^{M'}\prod_{i=1}^{n}(x+\alpha_{i})^{(\mu^{(i)})'_{1}}\left(D_V^{-1}\right)^{\#}\prod_{a=1}^{k}\left(\frac{d}{dx}-z_{a}\right)^{\lambda^{(a)}_{1}}.
\end{equation}
Clearly, $\Dti_V$ depends only on the reduced data $(\bar{\mu},\bar{\lambda};\bar{\alpha},\bar{z})^{\red}$.

\begin{thm}\label{1}
The following holds:
\begin{enumerate}
\item\label{1i}
$\Dti_V$ is a monic differential operator of order $\,L=\sum_{a=1}^{k}\lambda^{(a)}_{1}$.
\item\label{2i}
The vector space
\vvn.06>
$\,\tilde{V}=\ker\Dti_V$ is a space of quasi-exponentials with the data $(\bar{\lambda}',\bar{\mu}';\alb\bar{z},\alb-\bar{\alpha})$, where
\vvn.1>
$\,\bar{\mu}'=((\mu^{(1)})'\lc (\mu^{(n)})')$, $\,\bar{\lambda}'=((\lambda^{(1)})'\lc(\lambda^{(k)})')$ and $-\bar{\alpha} = (-\alpha_{1}\lc -\alpha_{n})$.
\item\label{3i}
Let $\,b_{ij}$ and $\,\tilde{b}_{st}$ be the coefficients of
\;$D_V^{\aug}$ and
\;$\Dti_V^{\aug}=\Dti_V\,\prod_{a=1,\,\lambda^{(a)}=0}^k(d/dx-z_a)$\,:
$$
D_V^{\aug}=\sum_{i=0}^{M'}\sum_{j=0}^{\infty}\,b_{ij}x^{-j}\left(\frac{d}{dx}\right)^{M'-i},\qquad
\Dti_V^{\aug}=\sum_{s=0}^{L}\sum_{t=0}^{\infty}\,\tilde{b}_{st}x^{-t}\left(\frac{d}{dx}\right)^{L-s}.\kern-2em
$$
Then there are polynomials \,$P_{st}$ in variables $r_{ij}\, $, $i=0\lc M'\?$,
$j\geq0$, depending only on the data
$(\bar{\mu},\bar{\lambda};\bar{\alpha},\bar{z})$, such that
\,$\tilde{b}_{st}\?$ equals the value of \,$P_{st}$ under the substitution
$\,r_{ij}=b_{ij}$ for all $\,i,j\,$. Moreover, the coefficients of \,$P_{st}$
are polynomials in \,$\bar{\alpha},\bar{z}$.
\end{enumerate}
\end{thm}

\noindent
The theorem will be proved in Section \ref{s6}.

\section{Bethe algebra}
\label{s4}
\subsection{Universal differential operator}
The current algebra $\gl_{n}[t]=\gl_{n}\otimes\C[t]$ is the Lie algebra of $\gl_{n}$-valued polynomials with pointwise commutator.
We identify the Lie algebra $\gl_{n}$ with the subalgebra $\gl_{n}\otimes 1$ of
constant polynomials in $\gl_{n}[t]$.

\vsk.2>
For each $g\in\gl_{n}$,
let $\,g(x)=\sum_{s=0}^{\infty}(g\otimes t^{s})\,x^{-s-1}$.
\vvn.1>
It is a formal power series in $x^{-1}$ with coefficients in $\gl_{n}[t]$.

For an $n\times n$ matrix $A$ with possibly noncommuting entries $a_{ij}$, its row determinant is
$$\rdet A =\sum_{\sigma\in S_{n}}(-1)^{\sigma}a_{1\sigma (1)}a_{2\sigma (2)}\dots a_{n\sigma (n)}.$$

Let $e_{ij}$, $i,j=1\lc n$, be the standard generators of the Lie algebra
$\gl_{n}$ satysfying the relations
$\,[e_{ij},e_{kl}]=\delta_{jk}e_{il}-\delta_{il}e_{kj}$.
Denote by $\mathfrak{h}$ the Cartan subalgebra of $\gl_{n}$ spanned by
the generators $e_{11}\lc e_{nn}$.

Fix $\bar{\alpha} =(\alpha_{1}\lc\alpha_{n})$, a sequence of pairwise distinct
complex numbers. Define the \textit{universal differential operator}
$\Dc_{\bar{\alpha}}$ by the formula
$$
\Dc_{\bar{\alpha}}=\rdet\left(\left(\frac{d}{dx}-\alpha_{i}\right)\delta_{ij}-
e_{ji}(x)\right)_{i,j=1}^{n}.
$$
It is a differential operator in the variable $x$ whose coefficients are formal power series in $x^{-1}$ with coefficients in $U(\gl_{n}[t])$,
\begin{equation}\label{B}
\Dc_{\bar{\alpha}}=\left(\frac{d}{dx}\right)^{n}+\sum_{i=1}^{n}B_{i}(x)\left(\frac{d}{dx}\right)^{n-i},
\end{equation}
where
\begin{equation}\label{B_{ij}}
B_{i}(x)=\sum_{j=0}^{\infty}B_{ij}x^{-j}.
\end{equation}
and $B_{ij}\in U(\gl_{n}[t])$ for $i=1\lc n$, $j\geq 0$.
Notice that $\,\sum_{i=0}^nB_{i0}u^{n-i}=\prod_{j=1}^n(u-\alpha_j)\,$.

\begin{defn}
The subalgebra $\mathcal{B}_{\bar{\alpha}}$ of $\,U(\gl_{n}[t])$ generated by $B_{ij}$,
$\,i=1\lc n$,
$\,j\ge 1$,
is called the Bethe algebra.
\end{defn}

The proof of the following theorem can be found in \cite{MTV4}.
\begin{thm}
The algebra $\mathcal{B}_{\bar{\alpha}}$ is commutative. The algebra $\mathcal{B}_{\bar{\alpha}}$ commutes with the subalgebra $U(\mathfrak{h})\subset U(\gl_{n}[t])$.
\end{thm}

\subsection{Action of the Bethe algebra in a tensor product of evaluation modules.}
\label{s42}
For $a\in\C$, let $\rho_{a}$ be the automorphism of $\gl_{n}[t]$ such that
$\rho_{a}:g(x)\mapsto g(x-a)$. Given a $\gl_{n}[t]$-module $M$, we denote by $M(a)$ the pullback of $M$ through the automorphism $\rho_{a}$.

Let $\ev:\gl_{n}[t]\to\gl_{n}$ be the evaluation homomorphism,
$\ev:g(x)\mapsto gx^{-1}$. For any $\gl_{n}$-module $M$, we denote by the same letter the $\gl_{n}[t]$-module, obtained by pulling $M$ back through the evaluation homomorphism. For each $a\in\C$ and $\gl_{n}$-module $M$, the $\gl_{n}[t]$- module $M(a)$ is called an \textit{evaluation module}.

For each $\lambda=(\lambda_{1}\lc\lambda_{n})\in\C^{n}$
and an $\hg$-module $M$,
we denote by $(M)_{\lambda}$ the weight subspace of $M$ of weight $\lambda$. Note that any partition $\lambda$ with $\lambda_{n+1}=0$ can be considered as an element of $\C^{n}$.

Let $M$ be a $\gl_{n}[t]$-module. As a subalgebra of $U(\gl_{n}[t])$, the algebra $\mathcal{B}_{\bar{\alpha}}$ acts on $M$. Since $\mathcal{B}_{\bar{\alpha}}$ commutes with $U(\mathfrak{h})$, it preserves the weight subspaces $(M)_{\lambda}$.

Given a
$\mathcal B_{\bar\alpha}$-module $M$,
a subspace $H\subset M$ is called an eigenspace of $\mathcal{B}_{\bar{\alpha}}$-action on $M$ if there is a homomorphism $\xi:\mathcal{B}_{\bar{\alpha}}\to\C\,$ such that $H=\bigcap_{F\in\mathcal{B}_{\bar{\alpha}}}\ker\bigl(F-\xi(F)\bigr)$.

Denote by $L_{\lambda}$ the irreducible finite-dimensional $\gl_{n}$-module with highest weight $\lambda$. Fix $\,\mu=(\mu_{1}\lc\mu_{n})\in\Z_{\geq 0}^{n}$, $\,\bar{\alpha}=(\alpha_{1}\lc\alpha_{n})\in\C^{n}$ such that $\alpha_{i}\neq\alpha_{j}$ for $i\neq j$, $\,\bar{z}=(z_{1}\lc z_{k})\in\C^{k}$ such that $z_{a}\neq z_{b}$ for $a\neq b$, and a sequence of partitions $\,\bar{\lambda}=(\lambda^{(1)}\lc\lambda^{(k)})$. Define the sequence of partitions $\bar{\mu}=(\mu^{(1)}\lc\mu^{(n)})$ setting $\mu^{(i)}=(\mu_{i},0,0,\dots{})$.
The next theorem states the results from \cite{MTV2} needed for the present paper.
\begin{thm}\label{2}
Consider a tensor product \,$L_{\bar\lambda}(\bar z)=
L_{\lambda^{(1)}}(z_{1})\otimes\ldots\otimes L_{\lambda^{(k)}}(z_{k})$
of evaluation $\gl_n[t]$-modules. Then the following holds.
\begin{enumerate}
\item Each eigenspace of the action of $\;\mathcal{B}_{\bar{\alpha}}$ on $\bigl(L_{\bar\lambda}(\bar z)\bigr)_{\!\mu}$ is one-dimensional.
\item\label{i2}
For generic $\bar{\alpha}$ and $\bar{z}$, the action of $\;\mathcal{B}_{\bar{\alpha}}$ on $\bigl(L_{\bar\lambda}(\bar z)\bigr)_{\!\mu}$ is diagonalizable.
\item\label{i3}
Let $v\in\bigl(L_{\bar\lambda}(\bar z)\bigr)_{\!\mu}$ be an
eigenvector of the action of $\;\mathcal{B}_{\bar{\alpha}}$. Then there exist
rational functions $\,b_1(x)\lc b_n(x)$, such that $B_{i}(x)v=b_{i}(x)v$ for all $\,i=1\lc n$,
and the kernel
of the differential operator $\;D=(d/dx)^{n}+\sum_{i=1}^{n}b_{i}(x)(d/dx)^{n-i}$
is the augmentation of a space of quasi-expo\-nen\-tials with the data
$(\bar{\mu},\bar{\lambda};\bar{\alpha},\bar{z})$.
\item\label{i4}
The correspondence between eigenspaces of the action of
$\;\mathcal{B}_{\bar{\alpha}}$ on $\bigl(L_{\bar\lambda}(\bar z)\bigr)_{\!\mu}$
and spaces of quasi-expo\-nen\-tials with the data
$(\bar{\mu},\bar{\lambda};\bar{\alpha},\bar{z})$ given in part \eqref{i3} is
bijective.
\end{enumerate}
\end{thm}

\subsection{Gaudin and Dynamical Hamiltonians}
\label{s43}
For $g\in\gl_{n}$, define
$g_{(a)}=1^{\otimes(a-1)}\!\otimes g\otimes 1^{\otimes(k-a)}\!
\in U(\gl_{n})^{\otimes k}$.
We will use the same notation for an element
of $U(\gl_{n})$ and its image under the
diagonal embedding
$\,g\mapsto\sum_{a=1}^{k}(g)_{(a)}\in U(\gl_{n})^{\otimes k}$. Let
$\Omega_{(ab)}=\sum_{i,j=1}^{n}(e_{ij})_{(a)}\,(e_{ji})_{(b)}\,$.

\smallskip
For sequences of pairwise distinct numbers $\bar{\alpha}=(\alpha_{1}\lc\alpha_{n})$ and $\bar{z}=(z_{1}\lc z_{k})$,
define the following elements of $U(\gl_{n})^{\otimes k}$:
$$
H_{a}(\bar{\alpha},\bar{z})=\sum_{i=1}^{n}\alpha_{i}(e_{ii})_{(a)}+
\sum_{\substack{b=1\\b\neq a}}^{k}\frac{\Omega_{(ab)}}{z_{a}-z_{b}}\;,\qquad
G_{i}(\bar{\alpha},\bar{z})=\sum_{a=1}^{k}z_{a}(e_{ii})_{(a)}+
\sum_{\substack{j=1\\j\neq i}}^{n}\frac{e_{ij}e_{ji}-e_{ii}}{\alpha_{i}-\alpha_{j}}\;.
$$
The elements $H_{1}(\bar{\alpha},\bar{z})\lc H_{k}(\bar{\alpha},\bar{z})$ are called the Gaudin Hamiltonians.
The elements $G_{1}(\bar{\alpha},\bar{z})\lc G_{n}(\bar{\alpha},\bar{z})$ are called the Dynamical Hamiltonians.

\smallskip
Consider
an algebra homomorphism $\ev_{\bar z}: U(\gl_{n}[t])\to U(\gl_{n})^{\otimes k}$, given by
$$\ev_{\bar z}: g\otimes t^{s}\,\mapsto\,\sum_{a=1}^{k}\,g_{(a)}z_{a}^{s}\,.$$
For each $i=1\lc n$, let $\Bh_{i}(x)$ be the image of the series $B_{i}(x)$,
see \eqref{B},
under the map $\ev_{\bar z}$.
\vv.1>
The series $\Bh_{i}(x)$ is a formal power series in $x^{-1}$ with coefficients in $U(\gl_{n})^{\otimes k}$. There
exists
a rational function of the form $\sum_{a=1}^{k}\sum_{j=0}^{i}\Bh_{ija}(x-z_{a})^{-j}$,
where
$\Bh_{ija}\in U(\gl_{n})^{\otimes k}$, such that $\Bh_{i}(x)$ is the Laurent series of this function as $x\to\infty$. We will identify the series $\Bh_{i}(x)$ and this rational function.

Let $\Ch_{j}(u)$, $\,j\in\Z_{\geq 0}$, be rational functions in $u$ defined by the formula
$$
\prod_{i=1}^{n}(u-\alpha_{i})\,\sum_{j=0}^{\infty}\Ch_{j}(u)x^{-j}\>=\,
u^{n}+\sum_{i=1}^{n}\Bh_{i}(x)u^{n-i}\>.
$$

\begin{lem}\label{KZinB}
The following holds:
$$
H_{a}(\bar{\alpha},\bar{z})\,=\,\Res_{x=z_{a}}
\Bigl(\frac{\Bh_{1}^{2}(x)}2-\Bh_{2}(x)\<\Bigr)\,,
\qquad
G_{i}(\bar{\alpha},\bar{z})=\Res_{u=\alpha_{i}}
\Bigl(\frac{\Ch_{1}^{2}(u)}2-\Ch_{2}(u)\<\Bigr)\,.
$$
\end{lem}
\begin{proof}
The proof is straightforward.
\end{proof}

\section{The $(\gl_{k},\gl_{n})$-duality}
\label{s5}
\subsection{The $(\gl_{k},\gl_{n})$-duality for Bethe algebras}
\label{s51}
Let $\mathfrak{X}_{n}$ be the
space of
all polynomials in anticommuting variables $\xi_{1}\lc\xi_{n}$.
Since $\xi_{i}\xi_{j}=-\xi_{j}\xi_{i}$ for any $\,i,j$, in particular,
$\xi_{i}^{2}=0$ for any $i$,
the monomials $\,\xi_{i_{1}}\dots\xi_{i_{l}}$,
$\,1\leq i_{1}<i_{2}<\ldots<i_{l}\leq n$, form
a basis of $\mathfrak{X}_{n}$.

The left derivations $\partial_1\lc\partial_n$ on $\mathfrak{X}_{n}$ are
linear maps such that
\begin{alignat}2
\label{leftderivation}
\partial_{i}\left(\xi_{j_{1}}\dots\xi_{j_{l}}\right) &{}=
(-1)^{s-1}\xi_{j_{1}}\dots\xi_{j_{s-1}}\,\xi_{j_{s+1}}\dots\xi_{j_{l}}\,,
\quad &&\text{if $\;i=j_s$ for some \,} s\,,
\\[2pt]
\notag
\partial_{i}\left(\xi_{j_{1}}\dots\xi_{j_{l}}\right) &{}=0\,,
&& \text{otherwise}\,.
\end{alignat}
It is easy to check that
$\partial_{i}\partial_{j}=-\partial_{j}\partial_{i}$ for any $\,i,j$,
in particular, $\partial_{i}^{2}=0$ for any $i$, and
$\partial_{i}\xi_{j}+\xi_{j}\partial_{i}=\delta_{ij}$ for any $\,i,j$.

\smallskip
Define a $\gl_{n}$-action on $\mathfrak{X}_{n}$ by
the rule
$e_{ij}\mapsto \xi_{i}\partial_{j}$. As a $\gl_{n}$-module, $\mathfrak{X}_{n}$ is isomorphic to $\bigoplus_{l=0}^{n}L_{\omega_{l}}$,
where
\begin{equation}
\label{omega}
\omega_{l}=(\underset{l}{\underbrace{ 1\lc 1}},0\lc 0)\,,
\end{equation}
and the component $\,L_{\omega_l}$ is spanned by the monomials of degree $\,l$.

Notice that the space $\mathfrak{X}_{n}$ coincides with the exterior algebra of $\C^{n}$.
The operators of left multiplication by $\xi_{1}\lc\xi_{n}$ and the left derivations $\partial_1\lc\partial_n$ give
on $\mathfrak{X}_{n}$
the irreducible representation of the Clifford algebra $\,\on{Cliff}_{n}$.

From now on, we will consider the Lie algebras $\gl_{n}$ and $\gl_{k}$ together.
We will write superscripts $\langle n\rangle$ and $\langle k\rangle$ to distinguish objects associated with algebras $\mathfrak{gl}_{n}$ and $\mathfrak{gl}_{k}$, respectively. For example, $e_{ij}^{\langle n\rangle}$, $\,i,j=1\lc n$,
are
the generators of $\mathfrak{gl}_{n}$, and $e_{ab}^{\langle k\rangle}$, $\,a,b=1\lc k$,
are
the generators of $\mathfrak{gl}_{k}$.

Let $\mathfrak{P}_{kn}$ be the vector space of polynomials in $kn$ pairwise anticommuting variables $\xi_{ai}$, $a=1\lc k$, $i=1\lc n$.
We have two vector space isomorphisms $\,\psi_{1}:(\mathfrak{X}_{n})^{\otimes k}\to \mathfrak{P}_{kn}$ and $\,\psi_{2} :(\mathfrak{X}_{k})^{\otimes n}\to \mathfrak{P}_{kn},$ given by:
\begin{align*}
\psi_{1} {}&{}: (p_{1}\otimes\ldots\otimes p_{k})\mapsto p_{1}(\xi_{11}\lc\xi_{1n})p_{2}(\xi_{21}\lc\xi_{2n})\dots p_{k}(\xi_{k1}\lc\xi_{kn})\,,
\\
\psi_{2} {}&{}: (p_{1}\otimes\ldots\otimes p_{n})\mapsto p_{1}(\xi_{11}\lc\xi_{k1})p_{2}(\xi_{12}\lc\xi_{k2})\dots p_{n}(\xi_{1n}\lc\xi_{kn})\,.
\end{align*}

Let $\,\partial_{ai}$, $\,a=1\lc k$, $\,i=1\lc n$, be the left derivations on $\mathfrak{P}_{kn}$ defined similarly to the left derivations
on $\mathfrak{X_{n}}$, see \eqref{leftderivation}.
Define actions of $\gl_{n}$ and $\gl_{k}$ on $\mathfrak{P}_{kn}$ by the formulas
$$e^{\langle n\rangle}_{ij}\mapsto \sum_{a=1}^{k}\xi_{ai}\partial_{aj}\,,\qquad e^{\langle k\rangle}_{ab}\mapsto \sum_{i=1}^{n}\xi_{ai}\partial_{bi}. $$
Then $\psi_{1}$ and $\psi_{2}$ are isomorphisms of $\gl_{n}$- and $\gl_{k}$-modules, respectively.

It is easy to check that $\gl_{n}$- and $\gl_{k}$-actions on $\,\mathfrak{P}_{kn}$ commute.
For the next
theorem, see for example \cite{CW}:
\begin{thm}
The $\gl_{n}\oplus \gl_{k}$-module $\mathfrak{P}_{kn}$ has the decomposition
$\;\mathfrak{P}_{kn}=\bigoplus_{\lambda}
V^{\langle n\rangle}_{\lambda}\otimes V^{\langle k\rangle}_{\lambda '}$,
where the sum runs over $\,\lambda=(\lambda_1\lc\lambda_n)$ such that
$\la_1\le k$.
\end{thm}

The $\gl_{n}$- and $\gl_{k}$-actions on $\mathfrak{P}_{kn}$ can be extended to the actions of corresponding current Lie algebras by the formulas
\begin{align}
\label{action1}
e^{\langle n\rangle}_{ij}\otimes t^{s} &{}\mapsto \sum_{a=1}^{k}\,z_{a}^{s}\,\xi_{ai}\partial_{aj},
\\[2pt]
\label{action2}
e^{\langle k\rangle}_{ab}\otimes t^{s} &{}\mapsto \sum_{i=1}^{n}\,(-\alpha_{i})^{s}\,\xi_{ai}\partial_{bi}.
\end{align}
Then $\psi_{1}$ and $\psi_{2}$ are
respective
isomorphisms of the following $\gl_{n}[t]$- and $\gl_{k}[t]$-modules:
\begin{gather}
\label{ps1z}
\psi_{1}:\mathfrak{X}_{n}(z_{1})\otimes\mathfrak{X}_{n}(z_{2})\otimes\ldots\otimes\mathfrak{X}_{n}(z_{k})\to\mathfrak{P}_{kn}\,,
\\[4pt]
\label{ps2a}
\psi_{2}:\mathfrak{X}_{k}(-\alpha_{1})\otimes\mathfrak{X}_{k}(-\alpha_{2})\otimes\ldots\otimes\mathfrak{X}_{k}(-\alpha_{n})\to\mathfrak{P}_{kn}\,.
\end{gather}

The actions \eqref{action1} and \eqref{action2} do not commute. Nevertheless, it turns out that the images of the subalgebras $B^{\langle n\rangle}_{\bar{\alpha}}$ and $B^{\langle k\rangle}_{\bar{z}}$ in $\End (\mathfrak{P}_{kn})$
given by these actions coincide. We will use Theorems \ref{1} and \ref{2} to show the following.

\begin{thm}\label{main}
Let
\vvn.1>
$\pi_{\bar{z}}^{\langle n\rangle}: U(\gl_{n}[t])\to \End (\mathfrak{P}_{kn})$ and $\pi_{-\bar{\alpha}}^{\langle k\rangle}: U(\gl_{k}[t])\to \End (\mathfrak{P}_{kn})$ be the homomorphisms defined by formulas \eqref{action1} and \eqref{action2}, respectively. Then
\begin{equation}
\pi_{\bar{z}}^{\langle n\rangle}(\mathcal{B}_{\bar{\alpha}}^{\langle n\rangle})=\pi_{-\bar{\alpha}}^{\langle k\rangle}(\mathcal{B}_{\bar{z}}^{\langle k\rangle}).
\end{equation}
\end{thm}

Theorem \ref{main}
is
proved in Section \ref{s55}.

\begin{rem}
Let $B^{\langle n\rangle}_{ij,\bar{\alpha}}$ be the generators of the algebra $\mathcal{B}_{\bar{\alpha}}^{\langle n\rangle}$, cf \eqref{B_{ij}}. Here we indicated the dependence on $\bar{\alpha}$, explicitly. Then we have $\pi^{\langle n\rangle}_{-\bar{z}}(B^{\langle n\rangle}_{ij,-\bar{\alpha}})=(-1)^{n-i-j}\pi^{\langle n\rangle}_{\bar{z}}(B^{\langle n\rangle}_{ij,\bar{\alpha}})$. Therefore $\pi^{\langle n\rangle}_{-\bar{z}}(\mathcal B^{\langle n\rangle}_{-\bar{\alpha}})=\pi^{\langle n\rangle}_{\bar{z}}(\mathcal B^{\langle n\rangle}_{\bar{\alpha}})$.
\end{rem}

\subsection{The $(\gl_{k},\gl_{n})$-duality for Gaudin and Dynamical Hamiltonians}
Define $U(\gl_{n})^{\otimes k}$ and $U(\gl_{k})^{\otimes n}$-actions on $\mathfrak{P}_{kn}$ by
\begin{equation}\label{action3}
(e^{\langle n\rangle}_{ij})_{(a)}\mapsto \xi_{ai}\partial_{aj},
\end{equation}
\begin{equation}\label{action4}
(e^{\langle k\rangle}_{ab})_{(i)}\mapsto \xi_{ai}\partial_{bi}.
\end{equation}
Then $\psi_{1}$ and $\psi_{2}$ are isomorphisms of $U(\gl_{n})^{\otimes k}$- and $U(\gl_{k})^{\otimes n}$-modules, respectively.

\smallskip
In Section \ref{s43}, we introduced elements $H_{a}(\bar{\alpha},\bar{z})$ and $G_{i}(\bar{\alpha},\bar{z})$ of $U(\gl_{n})^{\otimes k}$. We will write them now as $H_{a}^{\langle n,k\rangle}(\bar{\alpha},\bar{z})$, $G_{i}^{\langle n,k\rangle}(\bar{\alpha},\bar{z})$. We will also consider
analogous
elements $H_{i}^{\langle k,n\rangle}(\bar{z},\bar{\alpha})$, $G_{a}^{\langle k,n\rangle}(\bar{z},\bar{\alpha})$ of $U(\gl_{k})^{\otimes n}$. The following result can be found in \cite{TU}:

\begin{lem}
Let $\rho^{\langle n,k\rangle}: U(\gl_{n})^{\otimes k}\to \End (\mathfrak{P}_{kn})$ and $\rho^{\langle k,n\rangle}: U(\gl_{k})^{\otimes n}\to \End (\mathfrak{P}_{kn})$ be the homomorphisms defined by \eqref{action3} and \eqref{action4} respectively. Then for any $i=1\lc n$, and $a=1\lc k$ we have:
\begin{align}
\label{Duality for KZ}
\rho^{\langle n,k\rangle}\bigl(H_{a}^{\langle n,k\rangle}(\bar{\alpha},\bar{z})\bigr)&{}=-\rho^{\langle k,n\rangle}\bigl(G_{a}^{\langle k,n\rangle}(\bar{z},-\bar{\alpha})\bigr)\,,
\\[4pt]
\label{Duality for KZ 2}
\rho^{\langle n,k\rangle}\bigl(G_{i}^{\langle n,k\rangle}(\bar{\alpha},\bar{z})\bigr)&{}=\rho^{\langle k,n\rangle}\bigl(H_{i}^{\langle k,n\rangle}(\bar{z},-\bar{\alpha})\bigr)\,.
\end{align}
\end{lem}
\begin{proof}
The proof is straightforward.
\end{proof}

\subsection{Restriction to the subspaces $\mathfrak{P}_{kn}[\boldsymbol{l},\boldsymbol{m}]$.}
Let $\,\mathcal{Z}_{kn}$ be the subset of all pairs
$(\boldsymbol{l},\boldsymbol{m})\in\Z_{\geq 0}^{k}\times\Z_{\geq 0}^{n}\,$,
\,$\,\boldsymbol{l}=(l_{1}\lc l_{k})$, $\,\boldsymbol{m}=(m_{1}\lc m_{n})$,
such that $l_{a}\leq n$ for all $a$, $\,m_{i}\leq k$ for all $i$, and
$\sum_{a=1}^{k}l_{a}=\sum_{i=1}^{n}m_{i}$.
For each $(\boldsymbol{l}, \boldsymbol{m})\in\mathcal{Z}_{kn}$,
\vv.16>
denote by $\mathfrak{P}_{kn}[\boldsymbol{l},\boldsymbol{m}]\subset\mathfrak{P}_{kn}$ the span of all monomials $\xi_{11}^{d_{11}}\!\dots\xi_{k1}^{d_{k1}}\dots\,\xi_{1n}^{d_{1n}}\!\dots\xi_{kn}^{d_{kn}}$ such that $\sum_{a=1}^{k}d_{ai}=m_{i}$ and $\sum_{i=1}^{n}d_{ai}=l_{a}$.
\vv.1>
Note that $d_{ai}\in \{0,1\}$ for all $a,i$. Clearly, we have a vector space decomposition:
$$\mathfrak{P}_{kn}=\bigoplus_{(\boldsymbol{l}, \boldsymbol{m})\in \mathcal{Z}_{kn}}\mathfrak{P}_{kn}[\boldsymbol{l},\boldsymbol{m}]\,.$$

\begin{lem}
For any $(\boldsymbol{l}, \boldsymbol{m})\in \mathcal{Z}_{kn}$, the subspace $\mathfrak{P}_{kn}[\boldsymbol{l},\boldsymbol{m}]$ is invariant under the actions of the algebras $\;\mathcal{B}^{\langle n\rangle}_{\bar{\alpha}}$ and $\;\mathcal{B}^{\langle k\rangle}_{\bar{z}}$.
\end{lem}
\begin{proof}
Recall $\,\mathfrak{X}_{n}=\bigoplus_{l=0}^{n}L_{\omega_{l}}$ as
a $\gl_{n}$-module. Then by the isomorphism $\psi_1$, see \eqref{ps1z},
the $\mathfrak{gl}_n[t]$-module $\,\mathfrak{P}_{kn}$
is the direct sum of tensor products
$L_{\omega_{l_1}}^{\langle n\rangle}(z_{1})\otimes\ldots\otimes L_{\omega_{l_k}}^{\langle n\rangle}(z_{k})$, and
\begin{equation}
\label{isom1}
\mathfrak{P}_{kn}[\boldsymbol{l},\boldsymbol{m}]=\psi_1\bigl(
\bigl(L_{\omega_{l_1}}^{\langle n\rangle}(z_{1})\otimes\ldots\otimes L_{\omega_{l_k}}^{\langle n\rangle}(z_{k})\bigr)_{\boldsymbol{m}}\bigr)\,.
\end{equation}
Hence, $\,\mathfrak{P}_{kn}[\boldsymbol{l},\boldsymbol{m}]$ is invariant under
the action of $\,\mathcal{B}^{\langle n\rangle}_{\bar{\alpha}}$, see Section \ref{s42}.

Similarly, $\,\mathfrak{X}_{k}=\bigoplus_{m=0}^{k}L_{\omega_m}$ as
a $\gl_{k}$-module. Then by the isomorphism $\psi_2$, see \eqref{ps2a},
the $\mathfrak{gl}_k[t]$-module $\,\mathfrak{P}_{kn}$
is the direct sum of tensor products
$L_{\omega_{m_1}}^{\langle k\rangle}(-\alpha_{1})\otimes\ldots\otimes L_{\omega_{m_n}}^{\langle k\rangle}(-\alpha_n)$, and
\begin{equation}
\label{isom2}
\mathfrak{P}_{kn}[\boldsymbol{l},\boldsymbol{m}]=\psi_2\bigl(
\bigl(L_{\omega_{m_1}}^{\langle k\rangle}(-\alpha_{1})\otimes\ldots\otimes L_{\omega_{m_n}}^{\langle k\rangle}(-\alpha_n)\bigr)_{\boldsymbol{l}}\bigr)\,.
\end{equation}
Thus $\,\mathfrak{P}_{kn}[\boldsymbol{l},\boldsymbol{m}]$ is invariant under
the action of $\,\mathcal{B}^{\langle k\rangle}_{\bar z}$.
\end{proof}
We will prove Theorem \ref{main} by showing that the restrictions of $\pi_{\bar{z}}^{\langle n\rangle}(\mathcal{B}_{\bar{\alpha}}^{\langle n\rangle})$ and $\pi_{-\bar{\alpha}}^{\langle k\rangle}(\mathcal{B}_{\bar{z}}^{\langle k\rangle})$ to each subspace $\mathfrak{P}_{kn}[\boldsymbol{l},\boldsymbol{m}]$ coincide.
We will also need the following lemma.
\begin{lem}\label{simple spectrum}
Fix $(\boldsymbol{l},\boldsymbol{m})\in\mathcal{Z}_{kn}$.
For generic $\,\bar{\alpha},\bar{z}$, the common eigenspaces
of the operators $\,\rho^{\langle n,k\rangle}(H_{a}^{\langle
n,k\rangle}(\bar{\alpha},\bar{z}))$, $\;a=1\lc k$,
restricted to
$\,\mathfrak{P}_{kn}[\boldsymbol{l},\boldsymbol{m}]$ are one-dimensional.
Similarly, for generic $\,\bar{\alpha},\bar{z}$, the common eigenspaces
of the operators
$\rho^{\langle
k,n\rangle}(H_{i}^{\langle k,n\rangle}(\bar{z},-\bar{\alpha}))$, $\;i=1\lc n$,
restricted to $\,\mathfrak{P}_{kn}[\boldsymbol{l},\boldsymbol{m}]$
are one-dimensional.
\end{lem}
\begin{proof}
For every monomial $p\in\mathfrak{P}_{kn}$, we have
$(e^{\langle n\rangle}_{ii})_{(a)}p=m^{a}_{i}(p)p$ and $m^{a}_{i}(p)\in\Z$.
Moreover, if $p\neq
p'$, there exist $i,a$ such that $m^{a}_{i}(p)\neq m^{a}_{i}(p')$. Take
$\bar{\alpha}$ such that $\alpha_{1}\lc\alpha_{n}$ are linearly independent
over $\Z$. Then
for the operators $K_{a}=\rho^{\langle
n,k\rangle}(\sum_{i=1}^{n}\alpha_{i}(e^{\langle n\rangle}_{ii})_{(a)})$,
$a=1\lc k$, the common eigenspaces are one-dimensional.
Therefore, the common eigenspaces of the
operators $\rho^{\langle n,k\rangle}(H_{a}^{\langle
n,k\rangle}(\bar{\alpha},\bar{z}))=K_{a}+\sum_{b\neq
a}\Omega_{(ab)}(z_{a}-z_{b})^{-1}$, $a=1\lc k$,
restricted to a
finite-dimen\-sional submodule of $\mathfrak{P}_{kn}$ are one-dimensional
provided all the differences
$|z_{a}-z_{b}|$
are sufficiently large.
Hence,
for generic $\bar{\alpha}$ and $\bar{z}$,
the common eigenspaces of the operators
$\rho^{\langle n,k\rangle}(H_{a}^{\langle n,k\rangle}(\bar{\alpha},\bar{z}))$,
$a=1\lc k$,
restricted to a
$\mathfrak{P}_{kn}[\boldsymbol{l},\boldsymbol{m}]$ are one-dimensional.

The proof of the second claim is similar.
\end{proof}

\subsection{Spaces of quasi-exponentials and the $(\gl_{k},\gl_{n})$-duality}
Fix $(\boldsymbol{l},\boldsymbol{m})\in\mathcal{Z}_{kn}$,
and define
$\bar{\mu}=(\mu^{(1)}\lc\mu^{(n)}),\,
\bar{\lambda}=(\lambda^{(1)}\lc\lambda^{(k)})$ as follows.
If $\boldsymbol{l}=(l_{1}\lc l_{k})$ and $\boldsymbol{m}=(m_{1}\lc m_{n})$,
then $\mu^{(i)}=(m_{i},0,\dots{})$, $\,i=1\lc n$, and
$\lambda^{(a)}=\omega_{l_{a}}$, $\,a=1\lc k$, see \eqref{omega}.

By
Theorem \ref{2}
and formulas \eqref{isom1}, \eqref{isom2}, a space of quasi-exponentials with the data $(\bar{\mu},\bar{\lambda};\bar{\alpha},\bar{z})$
defined
above gives
rise to
an eigenvector of the action $\pi_{\bar{z}}^{\langle n\rangle}$ of $B_{\bar{\alpha}}^{\langle n\rangle}$ on $\mathfrak{P}_{kn}[\boldsymbol{l},\boldsymbol{m}]$. Similarly, a space of quasi-exponentials with the data $(\bar{\lambda}',\bar{\mu}';\bar{z},-\bar{\alpha})$ gives
rise to
an eigenvector of the action $\pi_{-\bar{\alpha}}^{\langle k\rangle}$ of $B_{\bar{z}}^{\langle k\rangle}$ on $\mathfrak{P}_{kn}[\boldsymbol{l},\boldsymbol{m}]$. We have the following theorem.

\begin{thm}\label{main2}
Let $V$ be a space of quasi-exponentials with the data $(\bar{\mu},\bar{\lambda};\bar{\alpha},\bar{z})$, and $v\in\mathfrak{P}_{kn}[\boldsymbol{l},\boldsymbol{m}]$ be an eigenvector of the action $\pi_{\bar{z}}^{\langle n\rangle}$ of $B_{\bar{\alpha}}^{\langle n\rangle}$ corresponding to $V$.
For the fundamental operator $\,D_{V}$ of the space $V$, define
the operator
$\,\Dti_{V}$ by formula \eqref{tilde D}, and set $\,\tilde{V}=\ker (\Dti_{V})$.
Then, for generic $\bar{\alpha},
\bar{z}$, the vector $v$ is an eigenvector of the action $\pi_{-\bar{\alpha}}^{\langle k\rangle}$ of $B_{\bar{z}}^{\langle k\rangle}$ corresponding to $\tilde{V}$.
\end{thm}

\begin{proof}
Let \,$b_{i}(x)$ and \,$\tilde{b}_{i}(x)$ be the coefficients of
\,$D_V^{\aug}=D_V\,\prod_{i=1,\,\mu^{(i)}=0}^n(d/dx-\alpha_i)$
\,and \,$\Dti_V^{\aug}=\Dti_V\,\prod_{a=1,\,\lambda^{(a)}=0}^k(d/dx-z_a)$\,:
$$
D_{V}^{\aug}=\left(\frac{d}{dx}\right)^{n}+
\sum_{i=1}^{n}b_{i}(x)\left(\frac{d}{dx}\right)^{n-i},
\qquad
\Dti_{V}^{\aug}=\left(\frac{d}{dx}\right)^{k}+
\sum_{a=1}^{k}\tilde{b}_{a}(x)\left(\frac{d}{dx}\right)^{k-a}\>.
$$
By Lemma \ref{coefficients},
$b_{i}(x)$ and $\tilde{b}_{i}(x)$ are rational functions of $x$. Define
functions $\,\tilde{c}_i(u)$, $\,i\in\Z_{\geq 0}$, by the rule:
$$\prod_{a=1}^{k}(u-z_{a})\,\sum_{i=0}^{\infty}\,\tilde{c}_{i}(u)x^{-i}=
u^{k}+\sum_{a=1}^{k}\,\tilde{b}_{a}(x)u^{k-a}.$$
Set
$$
h_{a}=\Res_{x=z_{a}}
\Bigl(\frac{b^{2}_{1}(x)}2-b_{2}(x)\<\Bigr),\qquad
\tilde{g}_{a}=\Res_{u=z_{a}}
\Bigl(\frac{\tilde{c}^{2}_{1}(u)}2-\tilde{c}_{2}(u)\<\Bigr).
$$
By a straightforward, though lengthy, calculation, one can show that
\begin{equation}\label{eqeigenvalues}
\tilde{g}_{a}=-h_{a}.
\end{equation}

By Lemma \ref{KZinB} and Theorem \ref{2},
for each $\,a=1\lc k$,
the vector $v$ is an eigenvector of $\rho^{\langle n,k\rangle}(H_{a}^{\langle n,k\rangle}(\bar{\alpha},\bar{z}))$ with eigenvalue $h_{a}$.
Similarly,
an eigenvector
$\tilde{v}\in\mathfrak{P}_{kn}[\boldsymbol{l},\boldsymbol{m}]$
of the action $\pi_{-\bar{\alpha}}^{\langle k\rangle}$ of $B_{\bar{z}}^{\langle k\rangle}$ corresponding to $\tilde{V}$
is an eigenvector of $\rho^{\langle k,n\rangle}(G_{a}^{\langle k,n\rangle}(\bar{z},-\bar{\alpha}))$ with eigenvalue $\tilde{g}_{a}$
for each $\,a=1\lc k$.
Therefore, by formulas \eqref{Duality for KZ} and \eqref{eqeigenvalues}
for each $\,a=1\lc k$, the vector
$\tilde{v}$ is
an eigenvector of $\rho^{\langle n,k\rangle}(H_{a}^{\langle n,k\rangle}(\bar{\alpha},\bar{z}))$ with eigenvalue $h_{a}$, the same as for $v$.
Hence, by Lemma \ref{simple spectrum},
the vector
$\tilde{v}$ is proportional to $v$.
\end{proof}

\subsection{Proof of Theorem \ref{main}}
\label{s55}
Let $B^{\langle n\rangle}_{ij,\bar{\alpha}}$, $i=1\lc n$, $j\in\Z_{\geq 0}$, and $B^{\langle k\rangle}_{st,\bar{z}}$, $s=1\lc k$, $t\in\Z_{\geq 0}$, be the generators of the algebras $\mathcal{B}_{\bar{\alpha}}^{\langle n\rangle}$ and $\mathcal{B}_{\bar{z}}^{\langle k\rangle}$, respectively, see
\eqref{B_{ij}}.

Assume first that $\bar{\alpha}$ and $\bar{z}$ are generic. Take a common eigenvector $v$ of $ \pi_{\bar{z}}^{\langle n\rangle}(B^{\langle n\rangle}_{ij,\bar{\alpha}})$, $\,i=1\lc n$, $\,j\in\Z_{\geq 0}$, corresponding to a space $V$ of quasi-exponentials with the data $(\bar{\mu},\bar{\lambda};\bar{\alpha},\bar{z})$ as in Theorem \ref{main2}. The eigenvalue of $ \pi_{\bar{z}}^{\langle n\rangle}(B^{\langle n\rangle}_{ij,\bar{\alpha}})$ associated to $v$ is the coefficient $b_{ij}$ in the expansion $D_{V}^{\aug}=\sum_{i=0}^{M'}\sum_{j=0}^{\infty}b_{ij}x^{-j}(d/dx)^{M'-i}$.
By Theorem \ref{main2}, $v$ is
also
a common eigenvector of $\pi_{-\bar{\alpha}}^{\langle k\rangle}(B^{\langle k\rangle}_{st,\bar{z}})$, $\,s=1\lc k$, $\,t\in\Z_{\geq 0}$, and the corresponding eigenvalue of $ \pi_{-\bar{\alpha}}^{\langle k\rangle}(B^{\langle k\rangle}_{st,\bar{z}})$ is the coefficient $\tilde{b}_{st}$ in the expansion $\Dti_{V}^{\aug}=\sum_{s=0}^{L}\sum_{t=0}^{\infty}\tilde{b}_{st}x^{-t}(d/dx)^{L-s}$, where $\Dti_{V}$ is given by the formula \eqref{tilde D}.
\vvn.16>
Due to Theorem \ref{1}, part \eqref{3i},
there exist polynomials $P_{st}$ in variables $r_{ij}$, $i=0\lc M'$, $j\geq 0$, independent of the eigenvector $v$, such that $\tilde{b}_{st}$
are obtained by
the substitution $r_{ij}=b_{ij}$ for all $i,j$, into the polynomial $P_{st}$,
\begin{equation}\label{relationfor coefficients}
\tilde{b}_{st}=P_{st}(\{b_{ij}\}).
\end{equation}

By Theorem \ref{2}, part \eqref{i4}, the subspace $\mathfrak{P}_{kn}[\boldsymbol{l},\boldsymbol{m}]$ has a basis consisting of common eigenvectors of the operators $\pi_{\bar{z}}^{\langle n\rangle}(B^{\langle n\rangle}_{ij,\bar{\alpha}})$, $\,i=1\lc n$, $\,j\in\Z_{\geq 0}\,$.
\vvn.06>
Since the operator $ \pi_{-\bar{\alpha}}^{\langle k\rangle}(B^{\langle k\rangle}_{st,\bar{z}})$ is diagonal in such a basis, relation \eqref{relationfor coefficients} for eigenvalues implies the analogous relation for the operators:
\begin{equation}\label{relation for generators}
\pi_{-\bar{\alpha}}^{\langle k\rangle}(B^{\langle k\rangle}_{st,\bar{z}})=P_{st}(\{\pi_{\bar{z}}^{\langle n\rangle}(B^{\langle n\rangle}_{ij,\bar{\alpha}})\}).
\end{equation}

Since
the operators $\pi_{-\bar{\alpha}}^{\langle k\rangle}(B^{\langle k\rangle}_{ab,\bar{z}})$, $\,\pi_{\bar{z}}^{\langle n\rangle}(B^{\langle n\rangle}_{ij,\bar{\alpha}})$, and the coefficients of $P_{st}$ depend polynomially on $\bar{\alpha}$ and $\bar{z}$,
relation
\eqref{relation for generators} holds for any $\bar{\alpha}$ and $\bar{z}$, and $\pi_{-\bar{\alpha}}^{\langle k\rangle}(\mathcal{B}_{\bar{z}}^{\langle k\rangle})\subset\pi_{\bar{z}}^{\langle n\rangle}(\mathcal{B}_{\bar{\alpha}}^{\langle n\rangle})$.

\smallskip
Exchanging the roles of $\gl_{k}$ and $\gl_{n}$, we obtain that $\pi_{-\bar{z}}^{\langle n\rangle}(\mathcal{B}_{-\bar{\alpha}}^{\langle n\rangle})\subset\pi_{-\bar{\alpha}}^{\langle k\rangle}(\mathcal{B}_{\bar{z}}^{\langle k\rangle})$ as well. Since $\pi_{-\bar{z}}^{\langle n\rangle}(\mathcal{B}_{-\bar{\alpha}}^{\langle n\rangle})=\pi_{\bar{z}}^{\langle n\rangle}(\mathcal{B}_{\bar{\alpha}}^{\langle n\rangle})$, see the remark at the end of Section \ref{s51}, Theorem \ref{main} is proved.

\section{Quotient differential operator}
\label{s6}
\subsection{Factorization of a differential operator}
For any functions $g_{1}\lc g_{n}$, let $$W(g_{1}\lc g_{n})=\det((g_{i}^{(j-1)})_{i,j=1}^{n})$$ be their Wronski determinant.
Let $W_{i}(g_{1}\lc g_{n})$ be the determinant of the $n\times n$ matrix whose $j$-th row is $g_{j}, g'_{j}\lc g_{j}^{(n-i-1)},$ $g_{j}^{(n-i+1)}\lc g_{j}^{(n)}$.

Consider a monic differential operator $D$ of order $n$ with coefficients $a_{i}(x)$, $i=1\lc n$:
\begin{equation}\label{D1}
D= \left(\frac{d}{dx}\right)^{n}+\sum_{i=1}^{n}a_{i}(x)\left(\frac{d}{dx}\right)^{n-i},
\end{equation}
and let $f_{1}, f_{2}\lc f_{n}$ be linearly independent solutions of the differential equation $Df=0$.

\begin{lem}\label{wronskian formula for differential op}
The coefficients $\,a_{1}(x)\lc a_{n}(x)$ of the differential operator $\,D$ are given by the formulas
\begin{equation}\label{a_{i}(x)}
a_{i}(x)=(-1)^{i}\,\frac{W_{i}(f_{1}\lc f_{n})}{W(f_{1}\lc f_{n})}\;, \qquad i=1\lc n\,,\kern-2em
\end{equation}
Moreover, for any function $g$,
we have
\begin{equation}\label{D}
Dg=\frac{W(f_{1}\lc f_{n}, g)}{W(f_{1}\lc f_{n})}\;.
\end{equation}
\end{lem}
\begin{proof}
The equations $Df_{1}=0\lc Df_{n}=0$ give a linear system of equations for the coefficients $a_{1}(x)\lc a_{n}(x)$. Solving this system by Cramer's rule yields formula \eqref{a_{i}(x)}. Formula \eqref{D} follows
from the last row expansion of the determinant in the numerator.
\end{proof}

\begin{prop}\label{fact}
The differential operator $\,D$ can be written in the following form:
\begin{equation}\label{D2}
D=\left(\frac{d}{dx} - \frac{g'_{1}}{g_{1}}\right)\left(\frac{d}{dx} - \frac{g'_{2}}{g_{2}}\right)\,\dots\,\left(\frac{d}{dx} - \frac{g'_{n}}{g_{n}}\right),
\end{equation}
where $g_{n}=f_{n}$, and
\begin{equation}
g_{i}=\frac{W(f_{n},f_{n-1}\lc f_{i} )}{W(f_{n}, f_{n-1}\lc f_{i+1})}\;,\qquad
i=1\lc n-1\,.\kern-.5em
\end{equation}
\end{prop}
\begin{proof}
Denote by $D_{1}$ the differential operator in the right hand side of \eqref{fact}. By Lemma \ref{wronskian formula for differential op} a monic differential operator is uniquely determined by its kernel. Therefore it is sufficient to prove that $D_{1}f_{i}=0$ for all $i=1\lc n$. We will prove it by induction on $n$.

If $n=1$, then $\,g_{1}=f_{1}$ and $\,D_{1}f_{1}=\left(d/dx-f'_{1}/f_{1}\right)f_{1}=0$.

Let $D_{2}$ be the monic differential operator whose kernel is spanned by $f_{2}\lc f_{n}$. By induction assumption, $$D_{2}=\left(\frac{d}{dx} - \frac{g'_{2}}{g_{2}}\right)\left(\frac{d}{dx} - \frac{g'_{3}}{g_{3}}\right)\,\dots\,\left(\frac{d}{dx} - \frac{g'_{n}}{g_{n}}\right).$$

Since $D_{1}=\left(d/dx-g'_{1}/g_{1}\right)D_{2}$, we have $D_{1}f_{i}=0$ for $i=2\lc n$.
\vvn.1>
Formula \eqref{D} yields $D_{2}f_{1}=g_{1}$, thus $D_{1}f_{1}=0$ as well.
\end{proof}

\subsection{Formal conjugate differential operator}
\label{s62}
Given a differential operator \eqref{D1}, define its formal conjugate by the formula: $$D^{\dagger}h(x)=\left(-\frac{d}{dx}\right)^{n}h(x)+\sum_{i=1}^{n}\left(-\frac{d}{dx}\right)^{n-i}\!\bigl(a_{i}(x)h(x)\bigr)\,.$$

Clearly, the formal conjugation is an antihomomorphism of the algebra of differential operators. In particular, if $\,D$ is given by formula \eqref{fact}, then
\begin{equation}\label{ddagger}
D^{\dagger}=(-1)^{n}\left(\frac{d}{dx}+\frac{g'_{n}}{g_{n}}\right)\left(\frac{d}{dx}+\frac{g'_{n-1}}{g_{n-1}}\right)\,\dots\,\left(\frac{d}{dx}+\frac{g'_{1}}{g_{1}}\right).
\end{equation}

\begin{prop}\label{Kernel for conjugate}
Let $$h_{i}\,=\,\frac{W(f_{1}\lc
f_{i-1},f_{i+1}
\lc f_{n})}{W(f_{1}\lc f_{n})}\;,$$
Then
the functions $\,h_{1}\lc h_{n}$ are linearly independent, and
$\,D^{\dagger}h_{i}=0$ for all $\,i=1\lc n$.
\end{prop}

\begin{proof}
Since $h_{1}=(-1)^{n-1}/g_{1}$, we have $D^{\dagger}h_{1}=0$ by formula \eqref{ddagger}.

Let $\sigma$ be a permutation of $\{1\lc n\}$. Take a new sequence $f_{\sigma (1)}\lc f_{\sigma (n)}$ of $n$ linearly independent solutions of the equation $Df=0$. Then similarly to the consideration above, we get
$$D^{\dagger}=(-1)^{n}\left(\frac{d}{dx}+\frac{g'_{n,\sigma}}{g_{n,\sigma}}\right)\left(\frac{d}{dx}+\frac{g'_{n-1,\sigma}}{g_{n-1,\sigma}}\right)\,\dots\,\left(\frac{d}{dx}+\frac{g'_{1,\sigma}}{g_{1,\sigma}}\right),$$
cf \eqref{ddagger}, where $g_{n,\sigma}=f_{\sigma (n)}$ and
$$g_{i, \sigma}=\frac{W(f_{\sigma (n)},f_{\sigma (n-1)}\lc f_{\sigma (i)})}{W(f_{\sigma (n)},f_{\sigma (n-1)}\lc f_{\sigma (i+1)})},\quad\quad i=1\lc n-1.$$
Taking $\sigma$ such that $\sigma (1) = i$, we get $D^{\dagger}h_{i}=0$.

To prove the linear independence of the functions $h_{1}\lc h_{n}$,
we will show that:
\begin{equation}
W(h_{1}\lc h_{n})\,=\,\frac{(-1)^{n(n-1)/2}}{W(f_{1}\lc f_{n})}\;.
\end{equation}

Let $\,p_{i}=W(f_{1}\lc
f_{i-1},f_{i+1}
\lc f_{n})$.
Denote by
$b_{ij}$ the $ij$-minor of the matrix $A=(f_{i}^{(j-1)})_{i.j=1}^{n}$.
Then we have $\,p_{i}=b_{in}$ and $\,p'_{i}=b_{i,n-1}$.

Since
$Df_i=0$ for any $\,i=1\lc n$,
we have $f_{i}^{(n)}=\sum_{l=1}^{n}a_{l}f_{i}^{(n-l)}$, where the functions
$a_1\lc a_n$ do not depend on $i$. Using this observation, one can check that
\begin{equation}
b'_{i,n-j}=b_{i,n-j-1}+(-1)^{j-1}a_{j+1}b_{in} + a_{1} b_{i,n-j}.
\end{equation}
Therefore,
by induction on $j$,
we have
\begin{equation}\label{p}
p_{i}^{(j)}=b_{i,n-j} + \sum_{k=0}^{j-1}C_{jk}b_{i,n-k}\,,\qquad i=1\lc n\,,
\kern-1em
\end{equation}
for certain functions $C_{jk}$, that do not depend on $i$.
Hence,
\begin{equation}
W(p_{1}\lc p_{n})=
\det(p_{i}^{(j)})_{\substack{\;\!i=1\lc n\hfill\\j=0\lc n-1}}=
\det(b_{i,n-j})_{\substack{\;\!i=1\lc n\hfill\\j=0\lc n-1}}
\end{equation}
and
\begin{align*}
W(h_{1}\lc h_{n})&{}=W\Bigl(\frac{p_{1}}{W(f_{1}\lc f_{n})}\lc\frac{p_{n}}{W(f_{1}\lc f_{n})}\Bigr)=
\frac{W(p_{1}\lc p_{n})}{(W(f_{1}\lc f_{n}))^{n}}
\\[6pt]
&{}=\frac{\det(b_{i,n-j})}{(W(f_{1}\lc f_{n}))^{n}}=
(-1)^{n(n-1)/2}\,\frac{\det((-1)^{i+j}\,b_{i,j})}{(W(f_{1}\lc f_{n}))^{n}}
\\[5pt]
&{}=(-1)^{n(n-1)/2}\,\frac{\det(A^{-1}\det A)}{(\det A)^{n}} = \frac{(-1)^{n(n-1)/2}}{W(f_{1}\lc f_{n})}\;.
\\[-32pt]
\end{align*}
\end{proof}
\subsection{Quotient differential operator}
\label{s63}
Let $D$ and $\Dh$ be differential operators
such that $\ker D\subset\ker \Dh $.
Then there is
a differential operator $\check{D}$, such that $\Dh=\check{D}D$. For instance, it can be seen from the factorization formula \eqref{D2}. We will call $\check{D}$ the \textit{quotient differential operator}.

Let $f_{1},f_{2}\lc f_{n}$ be a basis of $\,\ker D$ and
$f_{1},f_{2}\lc f_{n}, h_1\lc h_k$ be a basis of $\,\ker\hat D$.
Define functions $\phi_1\lc\phi_k$ by the formula
$$\phi_{a}=\frac{W(f_{1}\lc f_{n},h_{1}\lc
h_{a-1},h_{a+1}
\lc h_{k})}{W(f_{1}\lc f_{n},h_{1}\lc h_{k})}\;.
$$

\begin{prop}\label{kernel for quotient conjugate}
The functions $\phi_1\lc\phi_k$ are linearly independent, and
$\,\check{D}^{\dagger}\phi_{a}=0$ for all $a=1\lc k$.
\end{prop}
\begin{proof}
Set
$\,\tilde{h}_{a}=Dh_{a}$, $a=1\lc k$.
The functions $\tilde{h}_{1}\lc\tilde{h}_{k}$ are linearly independent. Indeed, if
there are
numbers $c_{1}\lc c_{k}$,
not all equal to zero, such that
$c_1\tilde h_1\lsym+c_k\tilde h_k=0$, then $D(c_1h_1\lsym+c_kh_k)=0$.
This means that $c_1h_1\lsym+c_kh_k$ belongs to the span of $f_{1}\lc f_{n}$
contrary to the linear independence of the functions
$f_{1}\lc f_{n},h_{1}\lc h_{k}$.

Formula \eqref{D} yields $\,\tilde{h}_{i}=W(f_{1}\lc f_{n},h_{i})/W(f_{1}\lc f_{n})$. Using identities for Wronski\-ans, see \cite{MV}, one can check that
$$\frac{W(\tilde{h}_{1}\lc
\tilde h_{a-1},\tilde h_{a+1}
\lc\tilde{h}_{k})}{W(\tilde{h}_{1}\lc\tilde{h}_{k})}=
\frac{W(f_{1}\lc f_{n},h_{1}\lc
h_{a-1},h_{a+1}
\lc h_{k})}{W(f_{1}\lc f_{n},h_{1}\lc h_{k})}=\phi_{a}\,.$$

Since $\check{D}\tilde{h}_{a}=\hat Dh_a=0$ for all $a=1\lc k$, the functions
$\tilde{h}_{1}\lc\tilde{h}_{k}$ form a basis of $\,\ker\check D$, because
the order of $\check{D}$ equals $k$. Since
$$\phi_{a}=\frac{W(\tilde{h}_{1}\lc
\tilde h_{a-1},\tilde h_{a+1}
\lc\tilde{h}_{k})}{W(\tilde{h}_{1}\lc\tilde{h}_{k})}\;,$$
Proposition \ref{kernel for quotient conjugate} follows from Proposition \ref{Kernel for conjugate} applied to $\check{D}$.
\end{proof}
\subsection{Quotient differential operator and spaces of quasi-exponentials}
\label{this}
Let $V$ be a space of quasi-exponentials with the data $(\bar{\mu},\bar{\lambda};\bar{\alpha},\bar{z})$.
For Section \ref{this} we will assume that
the data $(\bar{\mu},\bar{\lambda};\bar{\alpha},\bar{z})$ are reduced, that is,
the sequences $\bar\mu$ and $\bar\lambda$ do not contain zero partitions.
For each $i=1\lc n$,
denote $n_{i}=(\mu^{(i)})'_{1}$ and $p_{i}=\mu_{1}^{(i)}+n_{i}$.

Introduce also a larger space
$\Vh$ spanned by the functions $\,x^{p}e^{\alpha _{i}x}$ for all
$\,i=1\lc n$, and $p=0\lc p_{i}-1$. Denote
\begin{gather*}
W({\Vh})=W(e^{\alpha_{1}x},xe^{\alpha_{1}x}\lc x^{p_{1}-1}e^{\alpha_{1}x},
\,\ldots\,,\,
e^{\alpha_{n}x}, xe^{\alpha_{n}x}\lc x^{p_{n}-1}e^{\alpha_{n}x})\,,
\\[2pt]
W_{ij}(\Vh)=W({}\dots , \widehat{x^{j}e^{\alpha_{i}x}},\dots{})\,.
\end{gather*}
The functions in the second line are the same except
the function $x^{j}e^{\alpha_{i}x}$ is omitted.
\begin{lem}\label{denomnom} The following holds:
\begin{equation}\label{denom}
W(\Vh)=e^{\sum_{i=1}^{n}p_{i}\alpha_{i}x}
\,\prod_{i=1}^n\,\prod_{s=1}^{p_i-1}\,
s!\,\prod_{1\leq i<j\leq n}(\alpha_{j}-\alpha_{i})^{p_{i}p_{j}}\,,
\end{equation}
\begin{equation}\label{num}
W_{ij}(\Vh)=e^{\sum_{l=1}^{n}(p_{l}-\delta_{il})\alpha_{l}x}\,r_{ij}(x)
\,\prod_{l=1}^n\prod_{\substack{s=1\\(l,s)\ne(i,j)}}^{p_l-1}\!\!
s!\,\prod_{1\leq l<l'\leq n}(\alpha_{l'}-\alpha_{l})^{(p_{l}-\delta_{li})(p_{l'}-\delta_{l'i})}\,,
\end{equation}
where $r_{ij}(x)$ is a monic polynomial in $\,x$ and $\;\deg r_{ij}=p_{i}-j-1$.
\end{lem}
\begin{proof}
We will prove \eqref{denom} by induction on $\sum_{i=1}^{n}(p_{i}-1)=P$.
For $P=0$, equality \eqref{denom} becomes
$$W(e^{\alpha_{1}x},e^{\alpha_{2}x}\lc e^{\alpha_{n}x})=e^{\sum_{i=1}^{n}\alpha_{i}x}\prod_{1\leq i<j\leq n}(\alpha_{j}-\alpha_{i}),$$
which is equivalent to the Vandermonde determinant formula.

Fix $P_{0}\in\Z_{\geq 0}$. Suppose that \eqref{denom} is true for all $n$ and all $p_{1}\lc p_{n}$ such that $\sum_{i=1}^{n}(p_{i}-1)=P_{0}$. We will indicate the dependence of the space $\Vh$ on $p_{1}\lc p_{n}$ as follows: $\Vh^{p_{1}\lc p_{n}}$.

Fix $p_{1}\lc p_{n}$ such that $\sum_{i=1}^{n}(p_{i}-1)=P_{0}$. For each $l=1\lc n$, let $W_{(\beta,l)}$ be the Wronski determinant obtained from $W(\Vh^{p_{1}\lc p_{n}})$ by inserting
the exponential
$e^{\beta x}$ after the function $x^{p_{l}-1}e^{\alpha_{l} x}$.
Notice that
$(\partial/\partial\beta)^{p_{l}}|_{\beta=\alpha_{l}} W_{(\beta,l)}=W(\Vh^{p'_{1}\lc p'_{n}})$,
where $p'_{i}=p_{i}$ if $\,i\neq l$ and $p'_{l}=p_{l}+1$.

By the induction assumption, we have
$$
W_{(\beta,l)}\,=\,e^{\sum_{i=1}^{n}(p_{i}\alpha_{i}+\beta)x}
\,\prod_{i=1}^n\,\prod_{s=1}^{p_i-1}\,
s!\,\prod_{1\leq i<j \leq n}(\alpha_{j}-\alpha_{i})^{p_{i}p_{j}}\prod_{i=1}^{l}(\beta -\alpha_{i})^{p_{i}}\prod_{i=l+1}^{n}(\alpha_{i}-\beta)^{p_{i}},
$$
which gives
$$
\left.\left(\frac{\partial}{\partial\beta}\right)^{\!p_{l}}\right\rvert_{\beta=\alpha_{l}}\!\!W_{(\beta,l)}\,=\,
e^{\sum_{i=1}^{n}p'_{i}\alpha_{i}x}
\,\prod_{i=1}^n\,\prod_{s=1}^{p'_i-1}\,
s!\,\prod_{1\leq i<j\leq n}(\alpha_{j}-\alpha_{i})^{p'_{i}p'_{j}}\,.
$$
This proves the induction step
for formula \eqref{denom}.

To prove formula \eqref{num}, we fix $i$ and use induction on $s=p_{i}-j-1$.
The base of induction at $s=0$ is given by formula \eqref{denom}.

Fix $s_{0}\in\Z_{\geq 0}$. Suppose that \eqref{num} is true for all $n$, all $p_{1}\lc p_{n}$ and $j$ such that $s=s_{0}$. Fix $p_{1}\lc p_{n}$, and $j$ such that $p_{i}-j-1=s_{0}$. Let $W_{(\beta,i,j)}$ be the Wronski determinant obtained from $W_{ij}(\Vh^{p_{1}\lc p_{n}})$ by inserting
the exponential
$e^{\beta x}$ after the function $x^{p_{i}-1}e^{\alpha_{i} x}$ if $j\leq p_{i}-1$ or after the function $x^{p_{i}-2}e^{\alpha_{i} x}$ if $j=p_{i}-1$.
Notice that
$$(\partial/\partial\beta)^{p_{i}}|_{\beta=\alpha_{i}} W_{(\beta,i,j)}=W_{ij}(\Vh^{p'_{1}\lc p'_{n}})\,,$$
where $p'_{l}=p_{l}$ for $\,l\neq i$, $\,p'_{i}=p_{i}+1$, and $\,s'=p'_{i}-1-j=s_{0}+1$.

\smallskip
By the induction assumption, we have
\begin{align*}\label{ind2}
W_{(\beta,i,j)}\,=\,e^{\sum_{l=1}^{n}(p_{l}-\delta_{il})\alpha_{l}x+\beta x}
\,r_{ij}(x)
\, & \prod_{l=1}^n\prod_{\substack{s=1\\(l,s)\ne(i,j)}}^{p_l-1}\!\!
s!\,\prod_{1\leq l<l'\leq n}(\alpha_{l'}-\alpha_{l})^{(p_{l}-\delta_{li})(p_{l'}-\delta_{l' i})}
\\
{}\times{} &
\prod_{l=1}^{i}\,(\beta-\alpha_{l})^{p_{l}-\delta_{il}}
\prod_{l=i+1}^{n}(\alpha_{l}-\beta)^{p_{l}-\delta_{il}},
\end{align*}
where $r_{ij}(x)$ is a monic poynomial and $\,\deg r_{ij}(x) =p_{i}-j-1$.
The last formula gives
$$
\left. \displaystyle\left(\frac{\partial}{\partial\beta}\right)^{\!p_{i}}\right\rvert_{\beta=\alpha_{i}}\!\!W_{(\beta,i,j)}
\,=\,
e^{\sum_{l=1}^{n}(p'_{l}-\delta_{il})\alpha_{l}x}\,A(x)
\,\prod_{l=1}^n\prod_{\substack{s=1\\(l,s)\ne(i,j)}}^{p'_l-1}\!\!
s!\,\prod_{1\leq l<l'\leq n}(\alpha_{l'}-\alpha_{l})^{(p'_{l}-\delta_{li})(p'_{l'}-\delta_{l' i})} ,
$$
where $A(x)$ is a monic polynomial and $\,\deg A(x)=\deg r_{ij}(x)+1$.
This completes the induction step for formula \eqref{num}.
\end{proof}

For each $\,i=1\lc n$, set
\begin{equation}
\label{ddi}
\boldsymbol{d}_{i}=\{n_{i}+\mu^{(i)}_j-j\ |\ j=1\lc n_{i}\}\,,\qquad
\boldsymbol{d}_{i}^{c}=\{0,1,2\lc p_{i}-1\}\setminus\boldsymbol{d}_{i}\,.
\end{equation}

\begin{lem}\label{complementary space}
We have $\;\boldsymbol{d}_{i}^{c}=\{n_{i}-(\mu^{(i)})'_j+j-1\ |\ j=1\lc\mu_{1}^{(i)}\}\,$.
\end{lem}
\begin{proof}
Consider the Young diagram corresponding to the partition $\mu^{(i)}$.
Enumerate, starting from $0$, the sides of boxes in this diagram that form the bottom-right boundary, see the picture.

\medskip
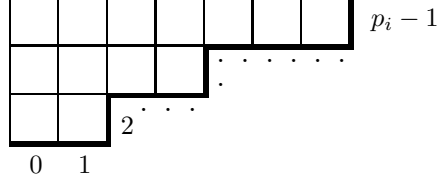
\hbox to\hsize{\hss\hsize130.2pt\parindent0pt
\vtop{\vskip-\baselineskip\vtop to 0pt{\offinterlineskip
\vrule height.4pt depth 0pt width 130.2pt\par
\vrule height18pt \kern18pt
\vrule height18pt \kern18pt
\vrule height18pt \kern18pt
\vrule height18pt \kern18pt
\vrule height18pt \kern18pt
\vrule height18pt \kern18pt
\vrule height18pt \kern17.8pt
\vrule height18pt width1.6pt\par
\vrule height.4pt depth0pt width73.8pt
\vrule height.4pt depth1.2pt width56.3pt\vskip-1.2pt
\vrule height18pt \kern18pt
\vrule height18pt \kern18pt
\vrule height18pt \kern18pt
\vrule height18pt \kern18pt
\vrule height18pt width1.6pt
\par
\vrule height.4pt depth0pt width36.8pt
\vrule height.4pt depth1.2pt width38.4pt
\vskip-1.2pt
\vrule height18pt \kern18pt
\vrule height18pt \kern18pt
\vrule height18pt width1.6pt
\par
\vrule height1.6pt width38.4pt
\vss}}\kern-130.2pt
\vtop{\vtop to 0pt{\strut\kern137pt\raise2.8pt\rlap{\footnotesize$p_i-1$}\kern-7pt\vskip-4pt
\strut\kern78.6pt\rlap{.}\kern9.2pt\rlap{.}\kern9.2pt\rlap{.}\kern9.2pt\rlap{.}\kern9pt\rlap{.}\kern9pt\rlap{.}\vskip-6pt
\strut\kern78.6pt\rlap{.}\vskip-5.6pt
\strut\kern49.4pt\rlap{.}\kern9.2pt\rlap{.}\kern9.2pt\rlap{.}\vskip-6.4pt
\strut\kern42.3pt\rlap{\footnotesize2}\vskip-.1pt
\strut\kern7.8pt\rlap{\footnotesize0}\kern18.4pt\rlap{\footnotesize1}\vss}
\vskip64pt}\hss}

\bigskip
\noindent
Then by \eqref{ddi}, the set $\,\boldsymbol d_i$
corresponds to the
right-most sides of the rows, which are
the vertical sides of the boundary.
Thus the complementary set $\,\boldsymbol d_i^c$
corresponds to the horizontal sides of the boundary,
which are the bottom sides of the columns.
The last observation proves the lemma.
\end{proof}

Let $D_{V}$ be the fundamental differential operator of $V$. Define $\Dh=\prod_{i=1}^{n}\left(d/dx-\alpha_{i}\right)^{p_{i}}$. Then $\ker\Dh=\Vh$. Therefore, $\,\ker D_{V}\subset\ker\Dh$, and
there exists a differential operator $\check{D}_{V}$, such that $\Dh=\check{D}_{V}D_{V}$,
see Section \ref{s63}. Let $\check{V}^{\dagger}=\ker\check{D}^{\dagger}$.

\begin{thm}\label{3}
The space
$\,\check{V}^{\dagger}$ is a space of quasi-exponentials with the data $(\bar{\mu}',\bar{\lambda}';-\bar{\alpha},\bar{z})$.
\end{thm}
\begin{proof}
The space $V$ has a basis of the form
$\,\{q_{ij}(x)e^{\alpha_{i}x}\ |\ i=1\lc n$, $\,j=1\lc n_{i}\}$, where $q_{ij}(x)$ are polynomials and $\,\deg q_{ij} = n_{i}+\mu^{(i)}_{j}-j\,$.
Then the functions $x^{l}e^{\alpha_{i}x}$, $\,i=1\lc n$, $\,l\in \boldsymbol{d}_{i}^{c}$, complement
this basis of $\,V$ to a basis of $\,\hat V$.

\smallskip
By Proposition \ref{kernel for quotient conjugate},
the space $\,\check{V}^{\dagger}$ has the following basis
\begin{equation}
\label{basVcd}
\frac{W_{ij}(\Vh)}{W(\Vh)}\,+\sum_{l=j+1}^{p_i-1}\,C_{ijl}\,\frac{W_{il}(\Vh)}{W(\Vh)}\;,
\qquad i=1\lc n\,,\quad j\in\boldsymbol{d}_{i}^{c}\,,\kern-3em
\end{equation}
where $\,C_{ijl}$ are complex numbers. Then by Lemma \ref{denomnom}, for
each $i,j$, the corresponding element of this basis has the form $\,\tilde r_{ij}(x)e^{-\alpha_{i}x}$,
where $\,\tilde r_{ij}(x)$ is a polynomial of degree $\,p_i-j-1$.

By Lemma \ref{complementary space}, $\,j\in\boldsymbol{d}_{i}^{c}$ if and only if
$\,j=n_{i}-(\mu^{(i)})'_{l}+l-1$ for some $\,l\in\{1\lc\mu_{1}^{(i)}\}\,$.
Set $\,\check q_{il}(x)=\tilde r_{ij}(x)$.
Then $\,\check V^\dagger$ has a basis of the form
$\,\{\check q_{il}(x)e^{-\alpha_{i}x}\ |\ i=1\lc n$, $\,l=1\lc\mu_1^{(i)}\}\,$
and
\begin{equation}
\label{degqc}
\deg\check q_{il}\,=\,\deg\tilde r_{ij}\,=\,
\mu_{1}^{(i)}+n_{i}-(n_{i}-(\mu^{(i)})'_{l}+l-1)-1=\mu_{1}^{(i)}+(\mu^{(i)})'_{l}-l\,.
\end{equation}

\smallskip
Recall $\,M'=\dim V=\sum_{i=1}^n(\mu^{(i)})'_1\,$.
Set $\,M=\dim\check V^\dagger=\sum_{i=1}^n\mu^{(i)}_1$.
We also have $\,\dim\hat V=M'+M$.

\smallskip
Fix a point $z\in \mathbb{C}$,
and
let $\boldsymbol{e}=\{e_{1}>\ldots>e_{M'}\}$ be the set of exponents of $V$ at $z$.
Then
there is
a basis $\{\psi_{1}\lc\psi_{M'}\}$ of $V$ such that
\begin{equation}
\label{psii}
\psi_{i}\,=\,(x-z)^{e_i}\bigl(1+o(1)\bigr)\,,\qquad x\to z\,,\kern-2em
\end{equation}
for any $i=1\lc M'$.

Set $\,\hat{\boldsymbol e}\,=\,\{\,\hat e_{1}<\hat e_{2}<\ldots<\hat e_{M}\}=
\{0,1,2\lc M'+M-1\}\setminus\boldsymbol e\,$.
By formula \eqref{denom}, the Wronskian $W(\Vh)$ has no zeros,
thus
$z$ is not a singular point of $\Vh$.
Therefore, there is
a basis $\{\psi_{1}\lc\psi_{M'},\chi_{1}\lc\chi_{M}\}$ of $\Vh$ such that
\begin{equation}
\label{chii}
\chi_{i}(x)\,=\,(x-z)^{\hat e_{i}}\bigl(1+o(1)\bigr)\,,\qquad
x\to z\,,\kern-2em
\end{equation}
for any $i=1\lc M$.

By Proposition \ref{kernel for quotient conjugate},
the set
\begin{equation}\label{basis}
\left\{\left.\frac{W(\psi_{1}\lc\psi_{M'},\chi_{1}\lc
\chi_{i-1},\chi_{i+1}
\lc\chi_{M})}{W(\psi_{1}\lc\psi_{M'},\chi_{1}\lc\chi_{M})}\ \;\right|\,\ i=1\lc M\right\}
\end{equation}
is a basis of $\check{V}^{\dagger}$.
Formulas \eqref{denom}, \eqref{psii}, \eqref{chii} show that
for any $\,i=1\lc M$,
$$
\frac{W(\psi_{1}\lc\psi_{M'},\chi_{1}\lc\chi_{i-1},\chi_{i+1}\lc
\chi_{M})}{W(\psi_{1}\lc\psi_{M'},\chi_{1}\lc\chi_{M})}\,=\,
C_i\,(x-z)^{M'+M-\hat e_i-1}\bigl(1+o(1)\bigr)
$$
as $x\to z$, where $\,C_i$ is a nonzero complex number.
Therefore, the set of exponents of $\check{V}^{\dagger}$ at the point $z$ is
$\,\check{\boldsymbol e}^\dagger=\{M'+M-\hat e_1-1\lsym>M'+M-\hat e_M-1\}$.
In particular, $z$ is a singular point of $\check{V}^{\dagger}$ if and only if
$z$ is a singular point of $V$.

\smallskip
If a partition $\lambda=(\lambda_{1},\lambda_{2},\dots{})$ corresponds
to the set $\boldsymbol e$, that is, $\lambda_{i}=e_i+i-M'$ for $i=1\lc M'$,
and $\lambda_{i}=0$ for $i>M'$, then similarly to
Lemma \ref{complementary space}, $\hat e_i=M'-\lambda_i'+i-1$, and
$\check e_i^\dagger=M'+M-\hat e_{i}-1=\lambda_{i}'+M-i$. Thus the set
$\check{\boldsymbol e}^\dagger$ of exponents of $\check{V}^{\dagger}$ at $z$
corresponds to a partition $\lambda'$.

Recall that the data $(\bar{\mu},\bar{\lambda};\bar{\alpha},\bar{z})$ are
reduced, in particular, $\bar z$ is the set of singular points of $V$.
To summarize, the consideration above shows that $\bar z$ is the set of
singular points of $\check{V}^{\dagger}$ as well, and $\check{V}^{\dagger}$
is the space of quasi-exponentials with the data
$(\bar{\mu}',\bar{\lambda}';-\bar{\alpha},\bar{z})$.
Theorem \ref{3} is proved.
\end{proof}

\subsection{Proof of Theorem \ref{1}}
It is sufficient to prove Theorem \ref{1}, parts \eqref{1i} and \eqref{2i} for the case of reduced data
$(\bar{\mu},\bar{\lambda};\bar{\alpha},\bar{z})$. This is immediate
for part \eqref{1i}, since $M',L,D_V$ and $\Dti_V$ depend only on
$(\bar{\mu},\bar{\lambda};\bar{\alpha},\bar{z})^{\red}$.
And for part \eqref{2i}, the following observation does the job:
if $(\bar{\mu},\bar{\lambda};\alb\bar{\alpha},\alb\bar{z})^{\red}=
(\bar{\mu}^{\red},\bar{\lambda}^{\red};\bar{\alpha}^{\red},\bar{z}^{\red})$,
then $(\bar{\lambda}',\bar{\mu}';\bar{z},-\bar{\alpha})^{\red}=
\bigl((\bar{\mu}^{\red})',\alb(\bar{\lambda}^{\red})';\bar{z}^{\red},-\bar{\alpha}^{\red}\bigr)$.

Let $V$ be a space of quasi-exponentials with the reduced data
$(\bar{\mu},\bar{\lambda};\bar{\alpha},\bar{z})$.
Let $\{f_{1}\lc f_{M'}\}$ be a basis of $V$ and
\vvn.1>
$D_{V}=\sum_{i=0}^{M'}b_{i}(x)(d/dx)^{M'-i}$ be the fundamental operator
of $V$.
For each $\,i=1\lc M'$, the ratio
$W_{i}(f_{1}\lc f_{M'})/W(f_{1}\lc f_{M'})$ is a rational function of $x$
regular at infinity. Together with
Lemma \ref{wronskian formula for differential op}, this proves
Lemma \ref{coefficients}, and we can consider $D_{V}$ as an invertible element
of $\Psi\mathfrak{D}$.

\smallskip
For any $i=0\lc M'$, let $\sum_{j=0}^{\infty}b_{ij}x^{-j}$ be the Laurent series of $b_{i}(x)$ at infinity. We will refer to the functions $b_{i}(x)$ as coefficients of the differential operator $D_{V}$, and to $b_{ij}$ as \textit{expansion} coefficients of the differential operator $D_{V}$. This terminology also applies to any differential operator with rational coefficients.

Notice that the formal conjugation $(\cdot)^{\dagger}$ of a differential operator, introduced in Section \ref{s62}, is consistent with the formal conjugation on $\Psi\mathfrak{D}$, introduced in Section \ref{s2}.
Recall the involutive antiautomorphism $(\cdot)^{\ddagger}: \Psi\mathfrak{D}\to \Psi\mathfrak{D}$ introduced in Section \ref{s2}.

\smallskip
Let $\,\Dh=\prod_{i=1}^{n}\left(d/dx-\alpha_{i}\right)^{\mu^{(i)}_1+(\mu^{(i)})'_1}$.
Denote by $\check{D}_{V}$
the quotient differential operator such that $\Dh=\check{D}_{V}D_{V}$.
Set $\,D_V^\times=\bigl(\,\prod_{a=1}^{k}\,(x-z_{a})^{\lambda^{(a)}_{1}}\check{D}_{V}^{\dagger}\bigr)^\ddagger$.
\vvn.1>
Recall the pseudodifferential operator $\tilde D_V$ defined by \eqref{tilde D}.
It is straightforward to verify that
\begin{equation}
\label{tildeD}
D_V^\times\,=\,(-1)^M\,\prod_{i=1}^{n}\,(x+\alpha_{i})^{\mu^{(i)}_{1}}\Dti_{V}\,,
\end{equation}
where $M=\mu^{(1)}_1\lsym+\mu^{(n)}_1\,$.

\smallskip
The next theorem is proved in \cite{MTV3}.

\begin{thm}\label{bispectral}
Let $\,D$ be the fundamental differential operator of a space of quasi-expo\-nen\-tials with the data
$(\bar{\mu}',\bar{\lambda}';-\bar{\alpha},\bar{z})$.
Then the following holds.
\begin{enumerate}
\item
\label{bsi}
The differential operator $\;\prod_{a=1}^{k}(x-z_{a})^{\lambda^{(a)}_1}D\,$
\vvn.1>
has polynomial coefficients.
\item
\label{bsii}
The differential operator
$\;\prod_{i=1}^{n}(x+\alpha_{i})^{-\mu^{(i)}_1}\bigl(\,\prod_{a=1}^{k}(x-z_{a})^{\lambda^{(a)}_1}D\bigr)^\ddagger$
\vvn.16>
is monic and has\\ order \,$L=\lambda^{(1)}_1\lsym+\lambda^{(k)}_1$.
\item
\label{bsiii}
The kernel of
$\;\bigl(\,\prod_{a=1}^{k}(x-z_{a})^{\lambda^{(a)}_1}D\bigr)^\ddagger$
is a space of quasi-expo\-nentials with the data\\
$(\bar{\lambda}', \bar{\mu}';\bar{z},-\bar{\alpha})$.
\end{enumerate}
\end{thm}

By Theorem \ref{3},
one can apply Theorem \ref{bispectral} to the monic differential operator
$(-1)^M\check{D}_{V}^{\dagger}$. Hence, the differential operator
$\prod_{a=1}^{k}(x-z_{a})^{\lambda^{(a)}_{1}}\check{D}_{V}^{\dagger}$
has polynomial coefficients
\vvn.1>
and the pseudodifferential operator
$D_V^\times$
is actually a differential operator.
Furthermore, formula \eqref{tildeD} and parts \eqref{bsii}, \eqref{bsiii} of Theorem \ref{bispectral}
yield partss \eqref{1i} and \eqref{2i} of Theorem \ref{1}.

\smallskip
To prove part \eqref{i3} of Theorem \ref{1}, consider a chain of
transformations:
\begin{equation}
\label{chain}
D_{V}^{\aug}\,\overset{(1)}{\longto}\,D_{V}\,\overset{(2)}{\longto}\,\check{D}_{V}\,\overset{(3)}{\longto}\,\check{D}_{V}^{\dagger}
\,\overset{(4)}{\longto}\,\prod_{a=1}^{k}\,(x-z_{a})^{\lambda^{(a)}_{1}}\check{D}_{V}^{\dagger}
\,\overset{(5)}{\longto}\,
D_V^\times\,\overset{(6)}{\longto}\,
\Dti_{V}\,\overset{(7)}{\longto}\,\Dti_{V}^{\aug}\,.\kern-1.6em
\end{equation}

\begin{lem}\label{transformation breakdown}
For each of the transformations in chain \eqref{chain}, the expansion coefficients of the transformed operator can be expressed as polynomials in the expansion coefficients of the initial operator.
\end{lem}
\begin{proof}
Fix $\beta\in\C$.
Let
$b_0(x)\lc b_{M'}(x)$, $b^{\beta}_0(x)\lc b^{\beta}_{M'+1}(x)$ be the coefficients of the differential operators $D_{V}$ and $D_{V}\left(d/dx-\beta\right)$:
$$D_{V}=\sum_{i=0}^{M'}b_{i}(x)\left(\frac{d}{dx}\right)^{M'-i},\quad D_{V}\left(\frac{d}{dx}-\beta\right)=\sum_{i=0}^{M'+1}b^{\beta}_{i}(x)\left(\frac{d}{dx}\right)^{M'+1-i}.$$
Then Lemma \ref{transformation breakdown} for transformation (1) follows from the relations:
\begin{equation}\label{transformation1}
b_{i}(x)=\sum_{j=0}^{i}\beta^{i-j} b^{\beta}_{j}(x),\quad i=1\lc M'.
\end{equation}

Let $c_0(x)\lc c_M(x)$, and $a_0\lc a_{M'+M}$,
be the coefficients of the differential operators $\check{D}_{V}$ and $\Dh$:
$$\check{D}_{V}=\sum_{j=0}^{M}c_{j}(x)\left(\frac{d}{dx}\right)^{M-j},\quad\Dh=\sum_{l=0}^{M'+M}a_{l}\left(\frac{d}{dx}\right)^{M'+M-l}.$$
The coefficients $a_{0}\lc a_{M'+M}$ are the elementary symmetric polynomials in $\alpha_{1}\lc\alpha_{n}$.

\smallskip
Fix $j=0\lc M$. Equalizing the coefficients for $(d/dx)^{M'+M-j}$ in both sides of the relation $\Dh=\check{D}_{V}D_{V}$, we get
\begin{equation}\label{relation for coefficients 2}
c_{j}(x)=a_{j}-\sum_{i=0}^{j-1}\sum_{l=0}^{i}c_{i-l}(x)\left(\frac{d^{l}}{dx^{l}}\,b_{j-i}(x)\right).
\end{equation}
Since
the function $c_{r}$ appears in the right-hand side of formula \eqref{relation for coefficients 2} only
for $r<j$,
we can recursively express $c_{j}(x)$ as polynomials in $b_{i}(x)$ and
their
derivatives. This proves the statement for transformation (2).

Let $\tilde c_{j}(x)$, $j=0\lc M$, be the coefficients of the differential operator $\check{D}_{V}^{\dagger}$:
$$\check{D}_{V}^{\dagger}=\sum_{j=1}^{M}\tilde c_{j}(x)\left(\frac{d}{dx}\right)^{M-j}.$$
Then we have $\,\tilde c_{j}(x)=\sum_{l=0}^{j}(-1)^{M-l}\left((d^{l}/dx^{l})c_{j-l}(x)\right)$.
This proves
the statement
for transformation (3).

\smallskip
For transformations (4), (6), and (7), the statement
is obvious.
For transformation (5), the statement
follows
from the definition of the antiautomorphism $(\cdot)^{\ddagger}$
that
transforms the coef\-ficients of a pseudodifferential operator $\sum_{i=-\infty}^{I}\sum_{j=-\infty}^{J}C_{ij}x^{i}(d/dx)^{j}$ by the rule $\,C_{ij}\mapsto C_{ji}$.
\end{proof}

Lemma \ref{transformation breakdown} provides an algorithm for expressing
the coefficients \,$\tilde b_{st}$ of the differential operator
$\Dti_{V}^{\aug}$ in item \eqref{i3} of Theorem \ref{1} via the coefficients
\,$b_{ij}$ of the operator $D_{V}^{\aug}$. It is clear that this algorithm
depends only on the data $(\bar{\mu},\bar{\lambda};\bar{\alpha},\bar{z})$ and
generates polynomial expressions in $\,b_{ij}$\,. This proves the existence of
the polynomials $P_{st}$ in item \eqref{i3} of Theorem \ref{1}.

It is easy to see that for each transformation in chain \eqref{chain},
expressions for expansion coefficients of the transformed operator in terms
of expansion coefficients of the initial operator
are polynomials in
$\bar{\alpha},\bar{z}$. For transformations (1) and (2)\:, it follows from
relations \eqref{transformation1} and \eqref{relation for coefficients 2},
respectively. Transformations (3) and (5) do not involve $\bar{\alpha}$ and
$\bar{z}$ at all. For transformations (4) and
\vvn.06>
(6)\:, notice that multiplication of a differential operator by the factor
$\,\prod_{a=1}^{k}(x-z_{a})^{\lambda^{(a)}_{1}}\!$ or
$\,\prod_{i=1}^{n}(x+\alpha_{i})^{-\mu^{(i)}_{1}}\!$ results in multiplication
of its expansion coefficients by polynomials in $\,z_1\lc z_k\,$ or
$\,\alpha_1\lc\alpha_n$, respectively. Finally, for transformation (7),
notice that for any $\beta\in\C$, multiplication of a differential operator
by $(d/dx-\beta)$ from the right results in multiplication of its expansion
coefficients by polynomials in $\beta$.

\smallskip
Theorem \ref{1} is proved.

\end{document}